\documentclass[11pt,reqno]{amsart}

\usepackage{amssymb, amsmath, amsthm}
\usepackage{hyperref}
\usepackage[alphabetic,lite]{amsrefs}
\usepackage{verbatim}
\usepackage{amscd}   
\usepackage[all]{xy} 
\usepackage{youngtab} 
\usepackage{young} 
\usepackage{ytableau}
\usepackage{tikz}
\usepackage{ mathrsfs }
\usepackage{cases}
\usepackage{array}
\usepackage{tabu}
\usepackage{calligra,mathrsfs}

\newcommand{\defi}[1]{{\upshape\sffamily #1}}
\DeclareMathOperator{\ShHom}{\mathscr{H}\text{\kern -3pt {\calligra\large om}}\,}

\newcommand{\D}{\mathcal{D}}

\newcommand{\bw}{\bigwedge}
\renewcommand{\det}{\textrm{det}}

\renewcommand{\ll}{\lambda}

\newcommand{\onto}{\twoheadrightarrow}
\newcommand{\oo}{\otimes}
\newcommand{\pd}{\partial}

\renewcommand{\SS}{\mathbb{S}}

\newcommand{\C}{\mathbb{C}}

\newcommand{\GL}{\operatorname{GL}}

\newcommand{\lie}{\mathfrak{g}}

\newcommand{\Pf}{\operatorname{Pf}}

\newcommand{\Spec}{\operatorname{Spec}}
\newcommand{\Sym}{\operatorname{Sym}}

\newcommand{\codim}{\operatorname{codim}}

\renewcommand{\det}{\operatorname{det}}
\newcommand{\dom}{\operatorname{dom}}

\renewcommand{\ker}{\operatorname{ker}}
\newcommand{\opmod}{\operatorname{mod}}
\newcommand{\op}{\operatorname}

\newcommand{\bb}[1]{\mathbb{#1}}

\renewcommand{\rm}[1]{\textrm{#1}}
\newcommand{\mc}[1]{\mathcal{#1}}
\newcommand{\mf}[1]{\mathfrak{#1}}
\newcommand{\ol}[1]{\overline{#1}}

\def\lra{\longrightarrow}

\newtheorem{theorem}{Theorem}[section]
\newtheorem*{theorem*}{Theorem}
\newtheorem*{problem*}{Problem}
\newtheorem{lemma}[theorem]{Lemma}

\newtheorem{proposition}[theorem]{Proposition}
\newtheorem{corollary}[theorem]{Corollary}
\newtheorem*{corollary*}{Corollary}

\newtheorem*{main-thm*}{Main Theorem}
\newtheorem*{linear-resolutions*}{Theorem on Linear Resolutions}
\newtheorem*{regularity-powers*}{Theorem on Regularity}
\newtheorem*{injectivity-Ext*}{Theorem on Injectivity of Maps of Ext Modules}
\newtheorem*{Kodaira*}{Kodaira Vanishing for Determinantal Thickenings}

\theoremstyle{definition}

\newtheorem*{definition*}{Definition}
\newtheorem{example}[theorem]{Example}

\theoremstyle{remark}
\newtheorem{remark}[theorem]{Remark}
\newtheorem*{remark*}{Remark}

\numberwithin{equation}{section}

\begin{document}

\title{Borel--Moore homology of determinantal varieties}

\author{Andr\'as C. L\H{o}rincz}
\address{Humboldt--Universit\"at zu Berlin, Institut f\"ur Mathematik, Berlin, Germany 12489}
\email{lorincza@hu-berlin.de}

\author{Claudiu Raicu}
\address{Department of Mathematics, University of Notre Dame, 255 Hurley, Notre Dame, IN 46556\newline
\indent Institute of Mathematics ``Simion Stoilow'' of the Romanian Academy}
\email{craicu@nd.edu}


\date{}


\begin{abstract} We compute the rational Borel--Moore homology groups for affine determinantal varieties in the spaces of general, symmetric, and skew-symmetric matrices, solving a problem suggested by the work of Pragacz and Ratajski. The main ingredient is the relation with Hartshorne's algebraic de Rham homology theory, and the calculation of the singular cohomology of matrix orbits, using the methods of Cartan and Borel. We also establish the degeneration of the \v Cech-de Rham spectral sequence for determinantal varieties, and compute explicitly the dimensions of de Rham cohomology groups of local cohomology with determinantal support, which are analogues of Lyubeznik numbers first introduced by Switala. Additionally, in the case of general matrices we further determine the Hodge numbers of the singular cohomology of matrix orbits and of the Borel--Moore homology of their closures, based on Saito's theory of mixed Hodge modules.
\end{abstract}

\maketitle

\section{Introduction}\label{sec:intro}

For an affine determinantal variety, it is well-known that both intersection homology and Chow homology are concentrated in even degrees, and the first calculations of these groups appear in work of Zelevinskii \cite{zelevinsky}*{Section~3.3} and Pragacz \cite{pragacz}*{Section~4}. By contrast, it was observed by Pragacz and Ratajski \cite{pragrat}*{Remark~2.4} that Borel--Moore homology can be nonzero in odd degrees, and hence that an explicit calculation of the groups is likely to be more subtle. The goal of this note is to completely determine the ranks of the Borel--Moore homology groups for determinantal varieties of general, symmetric and skew-symmetric matrices. Our approach combines classical methods for computing singular cohomology of homogeneous spaces, going back to the work of Cartan and Borel in the~50s, with the description of Borel--Moore homology via the algebraic de Rham homology theory introduced by Hartshorne in \cite{hartshorne-derham}. We obtain in addition several results of independent interest:
\begin{itemize}
    \item We establish the degeneration of the \v Cech-de Rham spectral sequence for determinantal (general, symmetric, and skew-symmetric) varieties. Such a degeneration statement is also known to hold for complete intersections by work of Hartshorne--Polini \cite{har-pol}, as well as for subspace arrangements and in small dimensions by work of Reichelt--Walther--Zhang \cite{RWZ}, but remains open in general (see \cite{switala-derham}*{Question~8.2} for the complete local case).
    \item We determine explicitly the de Rham cohomology groups of local cohomology with determinantal support, answering a question suggested to us by Switala. The dimensions of these groups are called the \defi{\v Cech-de Rham numbers} in \cite[Definition 1.2]{RWZ}.
    \item We describe the singular cohomology ring for the orbits of fixed rank matrices, following the work of Cartan \cite{cartan} and Borel \cite{borel} (see also \cite{zach}*{Proposition~3.6} for the case of general matrices).
    \item In the case of general matrices, we also determine the Hodge numbers associated to the mixed Hodge structures on the Borel--Moore homology of determinantal varieties and on the cohomology of matrix orbits. This is based on the weight filtration on local cohomology modules, determined in \cite{perlman2}. 
\end{itemize}

Before stating our results, we establish some notation and conventions. We study a matrix space $X$ with its rank stratification in the following three classical cases:
\begin{itemize}
\item[(a)] $X=\C^{m\times n}$ is the space of $m\times n$ matrices, $m\geq n$, and $O_p \subset X$ the set of matrices of rank $p$;
\item[(b)] $X=\bigwedge^2 \C^n$ is the space of $n\times n$ skew-symmetric matrices, and $O_p \subset X$ the set of matrices of rank $2p$;
\item[(c)] $X=\Sym^2 \C^n$ is the space of $n\times n$ symmetric matrices, and $O_p \subset X$ the set of matrices of rank $p$.
\end{itemize}

All the cohomology groups we consider have coefficients in $\bb{C}$. We write $H_i^{BM}(V)=H_i^{BM}(V,\bb{C})$ for the Borel--Moore homology (see \cite{bm}), and $H^i(V)=H^i(V,\bb{C})$ for the singular cohomology of a variety $V$, and write $h_i^{BM}(V)$ and $h^i(V)$ for their respective vector space dimensions. If we write $d_X$ for the (complex) dimension of the matrix space $X$ then we have
\begin{equation}\label{eq:BM-X}
\sum_{i\geq 0}h_i^{BM}(X)\cdot q^i = q^{2d_X}.
\end{equation}
To encode the Borel--Moore homology groups for the non-trivial orbit closures $\ol{O}_p\subsetneq X$, it is useful to introduce the \defi{$q$-binomial coefficients} ${a\choose b}_q$, which are polynomials in $\bb{Z}[q]$ defined for $a\geq b\geq 0$ by
\[{a\choose b}_q = \frac{(1-q^a)\cdot(1-q^{a-1})\cdots (1-q^{a-b+1})}{(1-q^b)\cdot(1-q^{b-1})\cdots (1-q)}.\]

\begin{theorem}\label{thm:main2} The Hilbert--Poincar\'e polynomials for the Borel--Moore homology groups of the orbit closures $\ol{O}_p\subsetneq X$ are given as follows.
\begin{itemize} 
\item[(a)] If $X=\bb{C}^{m\times n}$ and $m\geq n$, then
\[\sum_{i\geq 0} h_i^{BM}(\ol{O}_p) \cdot q^i \, = \, \sum_{s=0}^p q^{2s(m+n-s) + (p-s)(p-s+2)} \cdot {n \choose s}_{q^{-2}} \cdot {n-1-s \choose p-s}_{q^{2}}.\]
\item[(b)] If $X=\bigwedge^2 \bb{C}^n$ and $m=\lfloor n/2 \rfloor$, then
\[\sum_{i\geq 0} h_i^{BM}(\ol{O}_p) \cdot q^i \, = \, \sum_{s=0}^p q^{2s(2n-1-2s)+(p-s)(2p-2s+3)} \cdot {m \choose s}_{q^{-4}} \cdot {m-1-s \choose p-s}_{q^{4}}.\]
\item[(c)] If $X=\Sym^2 \C^n$, $m=\lfloor n/2 \rfloor$, and if we let
\begin{equation}\label{eq:def-eps-s} 
\epsilon_p = \begin{cases}
1 & \text{if }p\text{ is even and }n=2m+1\text{ is odd}, \\
0 & \text{otherwise},
\end{cases}
\end{equation}
then
\[\sum_{i\geq 0} h_i^{BM}(\ol{O}_p) \cdot q^i \, = \,
\sum_{\substack{s=0 \\ s 
\equiv p \!\!\!\! \pmod{2}}}^p \!
q^{s(2n+1-s)+\frac{(p-s)(p-s+3)}{2}} \cdot {m+\epsilon_p \choose 
\lfloor \frac{s}{2} \rfloor}_{q^{-4}} \cdot \binom{\lfloor\frac{n-s-1}{2}\rfloor}{\frac{p-s}{2}}_{q^{4}}.
\]
\end{itemize}
\end{theorem}

The reader may prefer to rewrite the formulas above using the identity 
\begin{equation}\label{eq:qbin-pos}
{a\choose b}_{q^{-1}} \, = \, q^{-b(a-b)}\cdot {a\choose b}_q.
\end{equation}
Our choice was made in order to connect more directly with the statement of Theorem~\ref{thm:main} below. To illustrate Theorem~\ref{thm:main2}, we consider some examples of orbit closures that are affine cones over familiar projective varieties.

\begin{example}\label{ex:p=1}
 We consider the case $p=1$, when $\ol{O}_1$ is the affine cone over a smooth projective variety ${\bf V}$.
 \begin{itemize}
     \item[(a)] If $X=\bb{C}^{m\times n}$ then ${\bf V}\simeq\bb{P}^{m-1}\times\bb{P}^{n-1}$ is a Segre product, and
     \[\sum_{i\geq 0} h_i^{BM}(\ol{O}_1) \cdot q^i = (q^3+q^5+\cdots+q^{2n-1}) + (q^{2m}+q^{2m+2}+\cdots+q^{2m+2n-2}).\]
     In particular, as noted in \cite{pragrat}*{Remark~2.3}, we have that $H^{BM}_3(\ol{O}_1)\neq 0$.
     \item[(b)] If $X=\bigwedge^2 \bb{C}^n$ then ${\bf V}\simeq\bb{G}(2,n)$ is a Grassmann variety, and if we let $m=\lfloor n/2\rfloor$ then
     \[\sum_{i\geq 0} h_i^{BM}(\ol{O}_1) \cdot q^i = (q^5+q^9+\cdots+q^{4m-3}) + (q^{4n-4m-2}+q^{4n-4m+2}+\cdots+q^{4n-6}).\]
     \item[(c)] If $X=\Sym^2 \bb{C}^n$ then ${\bf V}\simeq\nu_2(\bb{P}^{n-1})$ is the degree two Veronese embedding of $\bb{P}^{n-1}$, and
     \[\sum_{i\geq 0} h_i^{BM}(\ol{O}_1) \cdot q^i = q^{2n}.\]
 \end{itemize}
\end{example}

A key step in the proof of Theorem~\ref{thm:main2} is the calculation of the singular cohomology of the orbits $O_p$ of fixed rank matrices, which is based on general methods for computing cohomology of homogeneous spaces, pioneered by Cartan and Borel. The details, including the structure of the cohomology ring, are given in Section~\ref{sec:cohorb}, and in particular we get the following description for the ranks of the singular cohomology groups.

\begin{theorem}\label{thm:HP-orbits}
The Hilbert--Poincar\'e polynomials for the singular cohomology of the orbits $O_p\subset X$ are given as follows.
\begin{itemize}
\item[(a)] If $X=\bb{C}^{m\times n}$ and $m\geq n$, then
\[\sum_{i\geq 0} h^i(O_p)\cdot q^i = {n \choose p}_{q^2} \cdot (1+q^{2m-2p+1})\cdot (1+q^{2m-2p+3}) \cdots (1+q^{2m-1}).\]
\item[(b)] If $X=\bigwedge^2 \bb{C}^n$, $m=\lfloor n/2 \rfloor$, and if we let $\epsilon = n-2m$, then
\[\sum_{i\geq 0} h^i(O_p)\cdot q^i = {m \choose p}_{q^4} \cdot (1+q^{2(n+\epsilon)-4p+1})\cdot (1+q^{2(n+\epsilon)-4p+5}) \cdots (1+q^{2(n+\epsilon)-3}).\]
\item[(c)] Suppose that $X= \Sym^2 \C^n$, let $m=\lfloor n/2\rfloor$, and let $\epsilon = n-2m$. If $p=2r$ then
\[\sum_{i\geq 0} h^i(O_p)\cdot q^i = {m \choose r}_{q^4} \cdot (1+q^{2(n+\epsilon)-4r+1})\cdot (1+q^{2(n+\epsilon)-4r+5}) \cdots (1+q^{2(n+\epsilon)-3}).\]
If $p=2r+1$ then 
\[\sum_{i\geq 0} h^i(O_p)\cdot q^i = {m-1+\epsilon \choose r}_{q^4} \cdot [(1+q^{4m-4r+1})\cdot (1+q^{4m -4r+5}) \cdots (1+q^{4m-3})] \cdot (1+q^{2n-1}).\]
\end{itemize}
\end{theorem}

The relation between the invariants in Theorems~\ref{thm:main2} and~\ref{thm:HP-orbits} comes from the long exact sequence (see for instance \cite{pragrat}*{Lemma~2.2})
\begin{equation}\label{eq:les-BM-hom}
\cdots\lra H_i^{BM}(\ol{O}_{p-1})\overset{d_i}{\lra} H_i^{BM}(\ol{O}_{p})\lra H^{2d_{O_p}-i}(O_p)\lra H_{i-1}^{BM}(\ol{O}_{p-1})\overset{d_{i-1}}{\lra}H_{i-1}^{BM}(\ol{O}_{p})\lra\cdots
\end{equation}
where $d_{O_p}$ denotes the dimension of $O_p$. We then obtain inequalities
\begin{equation}\label{eq:HOp-leq-hBM}
h^{2d_{O_p}-i}(O_p) \leq h_i^{BM}(\ol{O}_p) + h_{i-1}^{BM}(\ol{O}_{p-1}),
\end{equation}
and note that equality holds for all $i$ if and only if the maps $d_i$ vanish for all $i$. Quite remarkably, this vanishing will occur most of the time.

\begin{theorem}\label{thm:di=0}
 The maps $d_i$ in \eqref{eq:les-BM-hom} vanish for all $i$ in the following cases:
 \begin{itemize}
     \item[(a)] $X=\bb{C}^{m\times n}$ and all $p$.
     \item[(b)] $X=\bw^2\bb{C}^{n}$ and all $p$.
     \item[(c)] $X=\Sym^2\bb{C}^{n}$ and $n-p$ even, or $p=1$.
 \end{itemize}
\end{theorem}

The following example shows that the assumption that $(n-p)$ is even is necessary when $X=\Sym^2\bb{C}^{n}$.

\begin{example}\label{ex:sym-n-p-odd}
 Suppose that $X=\Sym^2\bb{C}^{n}$, $p=2$, and $n=2m+1$. We have using Theorem~\ref{thm:HP-orbits}(c) that
 \[\sum_{i\geq 0} h^i(O_2)\cdot q^i = {m\choose 1}_{q^4}\cdot(1+q^{2n-1})=(1+q^4+\cdots+q^{4(m-1)})\cdot(1+q^{4m+1}),\]
 and in particular we have 
 \begin{equation}\label{eq:H4m=0}
     H^{2n-2}(O_2)=H^{4m}(O_2)=0.
 \end{equation}
 Moreover, by Theorem~\ref{thm:main2}(c) we have
 \[\sum_{i\geq 0} h_i^{BM}(\ol{O}_2) \cdot q^i = q^5\cdot{m\choose 1}_{q^4}+q^{2n}\cdot{m+1\choose 1}_{q^4} = (q^5+q^9+\cdots+q^{4m+1})+(q^{2n}+q^{2n+4}+\cdots+q^{2n+4m})\]
 and in particular we have $H_{2n}^{BM}(\ol{O}_2)=\bb{C}$. Recall from Example~\ref{ex:p=1} that $H_{2n}^{BM}(\ol{O}_1)=\bb{C}$, hence \eqref{eq:les-BM-hom} gives an exact sequence
 \[ \cdots \lra \bb{C} \overset{d_{2n}}{\lra} \bb{C} \lra H^{2d_{O_2}-2n}(O_2) \lra \cdots\]
 Using the fact that $d_{O_2}=2n-1$, we get $2d_{O_2}-2n=2n-2$, which combined with the vanishing \eqref{eq:H4m=0} shows that $d_{2n}$ is an isomorphism.
\end{example}

One can view \eqref{eq:HOp-leq-hBM} as a way to (collectively) bound from below the Borel--Moore homology of the orbit closures. For an upper bound, we study the \defi{\v Cech--de Rham spectral sequence} (using the terminology in~\cite{RWZ})
\begin{equation}\label{eq:CdR-sp-seq}
E_2^{ij}=H^i_{dR}(H^j_{\ol{O}_p}(\mc{O}_X)) \Longrightarrow H_{2d_X-i-j}^{BM}(\ol{O}_p),
\end{equation}
which follows by combining \cite{har-pol}*{Proposition~4.2} with the identification in \cite{har-pol}*{Theorem~3.1(7)} between Borel--Moore and de Rham homology. In \eqref{eq:CdR-sp-seq}, the groups $H^j_{\ol{O}_p}(\mc{O}_X)$ denote the local cohomology modules of the structure sheaf $\mc{O}_X$ with support in $\ol{O}_p$, which are regular holonomic $\mc{D}_X$-modules whose structure has been thoroughly analyzed in recent years \cites{raicu-weyman,raicu-weyman-loccoh,lor-rai,perlman}. For a $\mc{D}_X$-module~$M$, we denote by $H^i_{dR}(M)$ the cohomology groups of the (algebraic) de Rham complex
\begin{equation}\label{eq:drcomp}
DR(M): \qquad 0 \lra M \lra \Omega^1_X \oo_{\mc{O}_X} M \lra \cdots \lra \Omega^{d_X}_X \oo_{\mc{O}_X} M \lra 0,
\end{equation}
where $\Omega^i_X$ is the module of $i$-differential forms. The formation of de Rham cohomology $H^i_{dR}(M)$ agrees with the $\D$-module-theoretic derived integration (pushforward) $H^{i-d_X}(\pi_+ (M))$, where $\pi : X \to \{pt\}$ is the map to a point. It follows from \cite[Theorem 3.2.3]{htt} that if $M$ is holonomic then each $H^i_{dR}(M)$ is finite-dimensional, and this applies in particular to the groups $E_2^{ij}$ in \eqref{eq:CdR-sp-seq}. With the usual convention, we write $h^i_{dR}(M)$ for the vector space dimension of $H^i_{dR}(M)$. Note that although the Borel--Moore homology groups of $\ol{O}_p$ are intrinsic invariants (they to not depend on the embedding as a subvariety in $X$), the terms $E_2^{ij}$ in \eqref{eq:CdR-sp-seq} do a priori depend on both $\ol{O}_p$ and $X$. Quite remarkably, after an appropriate reindexing, they do provide intrinsic invariants of $\ol{O}_p$. More precisely, the \defi{\v Cech--de Rham numbers} (see \cite{RWZ}*{Section~2})
\begin{equation}\label{eq:def-CDR-numbers} 
\rho_{i,j}(\ol{O}_p) = h^{d_X-i}_{dR}(H^{d_X-j}_{\ol{O}_p}(\mc{O}_X))
\end{equation}
only depend on the variety $\ol{O}_p$ and not on the choice of the ambient affine space $X$: this was first proved by Switala over complete local rings \cite{switala-derham}*{Proposition~2.17}, and the version we use comes from \cite{bridgland}*{Theorem~1.1} (see also \cite[Theorem 6.2]{har-pol}).

Notice that the only non-vanishing \v Cech-de Rham number for $X$ is
\begin{equation}\label{eq:CdR-for-X}
\rho_{d_X,d_X} = h^0_{dR}(H^0_X(\mc{O}_X)) = 1,
\end{equation}
and in particular \eqref{eq:CdR-sp-seq} degenerates when $\ol{O}_p=X$, giving \eqref{eq:BM-X}. Our focus will therefore be on orbit closures $\ol{O}_p\subsetneq X$, where we have the following.

\begin{theorem}\label{thm:main}
The spectral sequence \eqref{eq:CdR-sp-seq} degenerates on the $E_2$ page for all the orbit closures $\ol{O}_p\subsetneq X$. Moreover, the bivariate generating functions for the \v Cech--de Rham numbers are given as follows.
\begin{itemize}
\item[(a)]
If $X=\bb{C}^{m\times n}$ and $m\geq n$, then
\[ \sum_{i,j\geq 0} \rho_{i,j}(\ol{O}_p)\cdot q^i \cdot w^j = \sum_{s=0}^p (qw)^{s(m+n-s)} \cdot {n \choose s}_{q^{-2}} \cdot w^{(p-s)(p-s+2)} \cdot {n-1-s \choose p-s}_{w^2}.\]
\item[(b)]
If $X=\bigwedge^2 \bb{C}^n$, $m=\lfloor n/2 \rfloor$, and if we let $\epsilon = n-2m$, then
\[ \sum_{i,j\geq 0} \rho_{i,j}(\ol{O}_p)\cdot q^i \cdot w^j  =\sum_{s=0}^p (qw)^{s(2n-1-2s)} \cdot {m \choose s}_{q^{-4}} \cdot w^{(p-s)(2p-2s+3)} \cdot {m-1-s \choose p-s}_{w^4}.\]
\item[(c)]
If $X=\Sym^2 \C^n$, $m=\lfloor n/2 \rfloor$, and if we take $\epsilon_p$ as in \eqref{eq:def-eps-s}, then 
\[ \sum_{i,j\geq 0} \rho_{i,j}(\ol{O}_p)\cdot q^i \cdot w^j  = 
 \sum_{\substack{s=0 \\ s 
\equiv p \!\!\!\! \pmod{2}}}^p \! (qw)^{\frac{s(2n+1-s)}{2}} \cdot {m+\epsilon_p \choose 
\lfloor \frac{s}{2} \rfloor}_{q^{-4}}
\cdot w^{\frac{(p-s)(p-s+3)}{2}} \cdot \binom{\lfloor\frac{n-s-1}{2}\rfloor}{\frac{p-s}{2}}_{w^4}.
\]
\end{itemize}
\end{theorem}

Notice that the degeneration of the spectral sequence \eqref{eq:CdR-sp-seq} is equivalent to the fact that the Euler--Poincar\'e polynomials in Theorem~\ref{thm:main2} are obtained from the generating functions in Theorem~\ref{thm:main} via the specialization~$w=q$. The expressions for the generating functions of \v Cech--de Rham numbers in Theorem~\ref{thm:main} illustrate the vanishing
\begin{equation}\label{eq:rho-vanishing}
\rho_{i,j}(\ol{O}_p) = 0\text{ for }i>j,
\end{equation}
which is established in general in \cite{RWZ}*{Proposition~2.1}. The inspiration for the study of \v Cech--de Rham numbers comes from the work of Lyubeznik \cite{lyubeznik}, where he defines using local cohomology groups a set of local invariants which are now usually referred to as \defi{Lyubeznik numbers}. There are many parallels between \v Cech--de Rham and Lyubeznik numbers, including the vanishing \eqref{eq:rho-vanishing}, and some are explored in \cite{RWZ}. In \cite{lor-rai} and \cite{perlman} the Lyubeznik numbers are computed for the determinantal varieties $\ol{O}_p$ in the spaces of general and skew-symmetric matrices, respectively, but they remain unknown in the case of symmetric matrices (see also the discussion in Sections~\ref{subsec:Lyub-comparison}, \ref{subsec:Lyub-comparison-skew}, and  \ref{subsec:Lyub-comparison-sym}).

\smallskip

As $O_p$ (and its closure) is a complex algebraic variety, the groups $H^i(O_p)$ and $H_i^{BM}(\ol{O}_p)$ are naturally endowed with mixed Hodge structures, by the work of Deligne (e.g. see \cite[Corollary 14.9]{mhs}). In general, a mixed Hodge structure $M$ carries an (increasing) weight $W_\bullet M$ and a (decreasing) Hodge $F^\bullet M$ filtration. The dimensions of the associated graded pieces are encoded by the \defi{Hodge numbers}
\[h^{p,q}(M) =\dim_{\C}{\op{Gr}}_{F}^{p}\,{\op{Gr}}_{p+q}^{W} \, M.\]

We say that the Hodge numbers of $M$ are \defi{concentrated on the diagonal} if $h^{p,q}(M)= 0$ whenever $p\neq q$. Note that in this case the weight filtration on $M$ determines all of its Hodge numbers, as for all $p$ we have $h^{p,p}(M) = \dim_\C \op{Gr}_{2p}^{W} M$, and further the vanishing $\op{Gr}_{2p+1}^{W} M=0$ must hold.

On the other hand, it follows from the work of Saito \cite{saito} that the local cohomology modules $H^i_{\ol{O}_p}(\mc{O}_X)$ naturally carry the structure of mixed Hodge modules. This has been studied in detail recently for the case (a) of general matrices by Perlman \cite{perlman2}. Based on his work, we compute the Hodge numbers of the singular cohomology of $O_p$ and Borel--Moore homology of $\ol{O}_p$ using the degeneration of the mixed Hodge module variant of the spectral sequence (\ref{eq:CdR-sp-seq}), together with Theorem \ref{thm:di=0} (a).

\begin{theorem}\label{thm:main-weight}
Let $X=\C^{m\times n}$ with $m\geq n$. The following bivariate generating functions record the weight filtrations on the mixed Hodge structures on $H_i^{BM}(\ol{O}_p)$ and $H^i(O_p)$, respectively:
\[\sum_{i,j\geq 0} \dim_{\C} \, \op{Gr}^W_j H^{BM}_i(\ol{O}_p)\cdot q^i \cdot w^j = \sum_{s=0}^p w^{p-s}\cdot (qw^{-1})^{2sm + (p-s)(p-s+2)}  \cdot {n \choose s}_{(qw^{-1})^{2}} \cdot {n-1-s \choose p-s}_{(qw^{-1})^2},\]

 \[\sum_{i,j\geq 0} \dim_{\C} \, \op{Gr}^W_j H^i(O_p)\cdot q^i \cdot w^j \, =\, {n \choose p}_{(qw)^2} \cdot \prod_{s=0}^{p-1} (1+q^{2m-2s-1} \cdot w^{2m-2s}).\]
Moreover, all of the corresponding Hodge numbers are concentrated on the diagonal.
\end{theorem}

These formulas yield Hodge-theoretic refinements to the ones in Theorem \ref{thm:main2} (a) and Theorem \ref{thm:HP-orbits} (a), respectively, which are recovered by evaluating $w\mapsto 1$. While this method of finding Hodge numbers works in principle also in the case of skew-symmetric and symmetric matrices, its implementation is contingent upon the determination of the weight filtration on the respective local cohomology modules, analogous to \cite{perlman2}. 

\medskip

\noindent{\bf Proof strategy.} We conclude this introduction with a summary of the strategy employed to prove the results presented here, the details of which are going to be explained in the rest of the paper.
\begin{enumerate}
    \item We describe the singular cohomology groups of the orbits $O_p$ using methods that go back to the classical work of Cartan and Borel, and obtain the formulas in Theorem~\ref{thm:HP-orbits}. This in particular gives an explicit formula for the total Betti numbers
    \begin{equation}\label{eq:btot-Op}
        b^{tot}(O_p) = \sum_{i\geq 0} h^i(O_p).
    \end{equation}

    \item Considering the total (Borel--Moore) Betti numbers,
    \begin{equation}\label{eq:btot-BM-Op-bar}
        b_{tot}^{BM}(\ol{O}_p) = \sum_{i\geq 0} h_i^{BM}(\ol{O}_{p}),
    \end{equation}
    we conclude using \eqref{eq:HOp-leq-hBM} that
    \begin{equation}\label{eq:btotOp-leq-btotBM}
        b^{tot}(O_p) \leq b_{tot}^{BM}(\ol{O}_p) + b_{tot}^{BM}(\ol{O}_{p-1}),
    \end{equation}
    with equality if and only if \eqref{eq:HOp-leq-hBM} is an equality for all $i$, which in turn is equivalent to the fact that the maps $d_i$ in the long exact sequence \eqref{eq:les-BM-hom} are zero for all $i$.
    
    \item  If we define the total \v Cech-de Rham numbers by
    \[\rho^{tot}(\ol{O}_p) = \sum_{i,j} \rho_{i,j}(\ol{O}_p) = \sum_{i,j}h^i_{dR}(H^j_{\ol{O}_p}(\mc{O}_X))\]
    then it follows from the spectral sequence \eqref{eq:CdR-sp-seq} that
    \begin{equation}\label{eq:btotBM-leq-btotdR}
        b_{tot}^{BM}(\ol{O}_p) \leq \rho^{tot}(\ol{O}_p)
    \end{equation}
    with equality if and only if the spectral sequence degenerates at the $E_2$ page.
    
    \item For each of the local cohomology modules $H^j_{\ol{O}_p}(\mc{O}_X)$, a composition series in the category of (equivariant) $\mc{D}_X$-modules is described in \cites{raicu-weyman,raicu-weyman-loccoh}, and for each of the simple composition factors, the corresponding de Rham cohomology groups are calculated in \cite{euler}. This provides an upper bound
    \[ \rho^{tot}(\ol{O}_p) \leq N_p\text{ for all }p,\]
    for certain explicit constants $N_p$, with equality if and only if the de Rham cohomology of each $H^j_{\ol{O}_p}(\mc{O}_X)$ is equal to the sum of the de Rham cohomology groups of its composition factors.
    
    \item We show that if $n-p$ is even then we have
    \begin{equation}\label{eq:btotOp=NpNpm1}
        b^{tot}(O_p) = N_p + N_{p-1},
    \end{equation}
    which implies that we must have equality throughout the chain of inequalities
    \[
    b^{tot}(O_p) \overset{\eqref{eq:btotOp-leq-btotBM}}{\leq} b_{tot}^{BM}(\ol{O}_p) + b_{tot}^{BM}(\ol{O}_{p-1})
    \overset{\eqref{eq:btotBM-leq-btotdR}}{\leq} \rho^{tot}(\ol{O}_p) + \rho^{tot}(\ol{O}_{p-1}) \leq N_p + N_{p-1}.
    \]
    In particular, we obtain the degeneration of the spectral sequence \eqref{eq:CdR-sp-seq} for all $p$, and get that the de Rham cohomology of local cohomology groups is the direct sum of the cohomologies of the composition factors, which is used to prove Theorem~\ref{thm:main}, and by specializing $w=q$, to prove Theorem~\ref{thm:main2}. Moreover, we get that \eqref{eq:HOp-leq-hBM} is an equality whenever $n-p$ is even, and in fact for all $p$ if $X=\bb{C}^{m\times n}$ or $X=\bw^2\bb{C}^n$, proving Theorem~\ref{thm:di=0}.
\end{enumerate}

\medskip

\noindent{\bf Organization.}  In Section~\ref{sec:prelim} we review some basic notation and techniques used to describe our computations, including aspects of de Rham cohomology, mixed Hodge structures, equivariant $\D$-modules, and representation theory of the general linear group. In Section \ref{sec:cohorb} we compute the singular cohomology groups of the orbits $O_p$. We then proceed to considering in more detail steps (2)--(5) of the strategy outlined above: for general matrices this is done in Section~\ref{sec:matrices}, for skew-symmetric matrices in Section~\ref{sec:skew}, and for symmetric matrices in Section~\ref{sec:DeRham-symmetric}. The results on mixed Hodge structures for the case of general matrices are proved in Section \ref{sec:mhs-det}. Finally, in Section \ref{sec:spec2} we discuss the degeneration of another spectral sequence, that is closely related to (\ref{eq:CdR-sp-seq}).

\section{Preliminaries}\label{sec:prelim}

Throughout this section $X$ is an irreducible smooth complex affine variety. We freely identify $\mc{O}_X$-modules with their global sections. We always work with left $\D$-modules.

\subsection{De Rham cohomology}\label{subsec:derham}
The (analytic) de Rham complex for $\D$-modules plays a fundamental role in the Riemann-Hilbert correspondence (for example, see \cite[Theorem 7.2.5]{htt}). In the special case when $M=\mc{O}_X$ is the structure sheaf, the celebrated comparison theorem of Grothendieck \cite{grothendieck} implies that the space $H^i_{dR}(\mc{O}_X)$ agrees with the (singular) cohomology group $H^i(X,\C)$. More generally, for an irreducible closed subvariety $Y\subset X$, the local cohomology group $H^{\codim_X Y}_Y(\mc{O}_X)$ has a unique simple submodule $\mc{L}(Y,X)$ (called the Brylinski-Kashiwara module \cite[Section 8]{bry-kash}) whose associated de Rham complex is the (middle perversity) intersection cohomology sheaf of $Y$. Hence, the de Rham cohomology groups of $\mc{L}(Y,X)$ agree with the intersection cohomology groups of $Y$ (for example, this follows from \cite[Theorem 7.1.1]{htt}).

In contrast with de Rham cohomology (see discussion after (\ref{eq:drcomp})), the Lyubeznik numbers mentioned in the Introduction can be understood as the (derived) restriction to the origin of the local cohomology modules. But pushforward of a module $M$ from an \defi{affine space} to the origin is the same as the restriction to the origin of its Fourier transform $\mc{F}(M)$ (see \cite[Proposition 3.2.6]{htt}):
\begin{equation}\label{eq:pushpull}
H^{k}(\pi_+ (M)) \cong H^{k}(Li^*\mc{F}(M)),
\end{equation}
where $\pi : X \to \{0\}$ is the projection and $i : \{0\} \to X$ the inclusion. While the latter uses only the $S=\C[x_1,\dots,x_{d_X}]$-module structure of $M$, the former uses only its $\C[\pd_1,\dots , \pd_{d_X}]$-structure, as can be seen also from the differentials in the de Rham complex
\begin{equation}\label{eq:diff-derham}
d^i: M\oo \Omega^i_X \to M\oo \Omega^{i+1}, \quad d^i(m \, dx_{j_1}\wedge\dots \wedge dx_{j_i})=\sum_{s=1}^{d_X} \pd_s (m) \, dx_s \wedge dx_{j_1}\wedge\dots \wedge dx_{j_i}.
\end{equation}
Hence, in a sense we should expect our calculations regarding the \v Cech--de Rham numbers to reflect features dual to those encoded by the Lyubeznik numbers. We explain in detail why this is indeed the case for our spaces of matrices in Sections \ref{subsec:Lyub-comparison}, \ref{subsec:Lyub-comparison-skew}, and  \ref{subsec:Lyub-comparison-sym}.

\subsection{Mixed hodge structure on de Rham cohomology}\label{sec:mhs}

As mentioned, de Rham cohomology can be interpreted as (derived) pushforward to a point. Thus, if $M$ is a mixed Hodge module, as developed by Saito \cite{saito}, then $H^i_{dR}(M)$ naturally carries a mixed Hodge structure for any $i$.

If $M$ is a mixed Hodge module, we denote by $M(k)$ its $k$th Tate twist, that shifts weights by $-2k$.

We denote by $\mc{O}_X^H$ the constant (trivial) mixed Hodge module on $X$, for which the graded pieces of the weight filtration give the $\D_X$-modules $\op{Gr}^W_{d_X} \mc{O}_X^H = \mc{O}_X$, and $\op{Gr}^W_k \mc{O}_X^H = 0, k \neq d_X$.

Let $Z\subsetneq X$ be a closed subvariety, $U=X\setminus Z$ the complement, and $\iota: U \to X$ the open embedding. Since 
\begin{equation}\label{eq:locpush}
H^1_Z(\mc{O}_X)\, \cong \, \iota_* \mc{O}_U /\mc{O}_X, \quad \mbox{and} \quad H^i_Z(\mc{O}_X) \, \cong \, \mathbf{R}^{i-1}\iota_*(\mc{O}_U), \quad \mbox{for all } i\geq 2,
\end{equation}
the local cohomology modules $H^j_Z(\mc{O}_X^H)$ naturally carry mixed Hodge module structures for all $j$ (cf. \cite{saito}).

In conclusion, de Rham cohomology of local cohomology $H^i_{dR}(H^j_Z(\mc{O}_X^H))$ acquires also a mixed Hodge structure for all $i,j$. Furthermore, Borel--Moore homology $H_i^{BM}(Z)$ carries mixed Hodge structure as well, for all $i$ \cite[Corollary 14.9]{mhs}. The following relates these mixed Hodge structures through the spectral sequence (\ref{eq:CdR-sp-seq}).

\begin{proposition}\label{prop:mhs}
Let $X=\C^d$ and $Z\subsetneq X$ a closed subvariety. The \v Cech--de Rham spectral sequence
\[H^i_{dR}(H^j_{Z}(\mc{O}_X^H)) \, \Longrightarrow \,\, H_{2d-i-j}^{BM}(Z)(-d)\]
is a spectral sequence of mixed Hodge modules.
\end{proposition}

\begin{proof}
We reinterpret the spectral sequence using the identifications in (\ref{eq:locpush}) as follows. Let $\pi_U: U\to \{pt\}$ and $\pi_X: X \to \{pt\}$ denote maps to a point. The (higher) pushforward of $\mc{O}_U^H$ via $\pi_U$ yields cohomology of $U$, and factoring this through $\pi_U = \pi_X \circ \iota$ yields the following spectral sequence of mixed Hodge modules (cf. \cite[Section 14.1.3]{mhs})
\begin{equation}\label{eq:specu}
H^i_{dR}(\mathbf{R}^j \iota_* \mc{O}_U^H) \, \Longrightarrow \, H^{i+j}(U).
\end{equation}
Since $U$ is smooth, we have as mixed Hodge structures (see \cite[Corollary 6.26]{mhs})
\begin{equation}\label{eq:poinbm}
H^k(U) \, \cong \, H^{BM}_{2d-k}(U)(-d).
\end{equation}
By the long exact sequence in Borel--Moore homology corresponding to $\iota: U \to X$ (analogous to (\ref{eq:les-BM-hom})) and (\ref{eq:BM-X}), we obtain
\begin{equation}\label{eq:bmzu}
    H^{BM}_i(Z) \cong H_{i+1}^{BM}(U), \mbox{ for } i\leq 2d-2, \quad H^{BM}_{2d}(U)=\C.
\end{equation}
Note that $H_{dR}^0(\mc{O}_X)=\C$ and $H_{dR}^i(\mc{O}_X)=0$ when $i>0$. Further, from (\ref{eq:specu}) we get $H^0_{dR}(\iota_* \mc{O}_U)=\C$. Applying de Rham cohomology to the exact sequence of mixed Hodge modules
\[0 \to \mc{O}^H_X \to \iota_* \mc{O}_U^H \to H_Z^1(\mc{O}^H_X) \to 0,\] 
we obtain
\[H^i_{dR}(H_Z^1(\mc{O}_X^H)) \cong H^i_{dR}(\iota_* \mc{O}_U^H), \mbox{ for } i\geq 1, \quad H^0_{dR}(H_Z^1(\mc{O}_X^H)) = 0.\] 
Using this together with (\ref{eq:locpush}), (\ref{eq:poinbm}), (\ref{eq:bmzu}), we obtain the desired spectral sequence from the one in (\ref{eq:specu}).
\end{proof}

\subsection{Equivariant $\D$-modules}\label{sec:equivd}

Here we provide some background on equivariant $\D$-modules. For more details, see \cite{lor-wal}.

Let a connected algebraic group $G$ act on $X$. A (possibly infinite-dimensional) vector space $V$ is a rational $G-$module, if $V$ is equipped with a linear action of $G$, such that every $v\in V$ is contained in some finite-dimensional $G$-stable subspace $W\subset V$ with the map $G \to \GL(W)$ being a morphism of algebraic varieties.

We call $M$ a (strongly) $G$-\defi{equivariant} $\D$-module, if we have a $\D_{G\times X}$-isomorphism
\[
\tau\colon p^*M \rightarrow m^*M,
\]
where $p$ and $m$ are the projection and multiplication maps
\[
p\colon G\times X\to X,\qquad\qquad m\colon G\times X\to X
\]
respectively, and $\tau$ satisfies the usual compatibility conditions on $G\times G\times X$ (see \cite[Definition 11.5.2]{htt}).

Let $\lie$ be the Lie algebra of $G$. Differentiating the action of $G$ on $X$ yields a map $\lie \to \D_X$. Equivariance of a $\D$-module $M$ amounts to $M$ having a rational $G$-module structure such that differentiating the action of $G$ coincides with the action of $\lie$ induced from $\lie \to \D_X$.

We denote by $\op{mod}(\D_X)$ the category of coherent $\D_X$-modules, and its subcategory of coherent equivariant $\D$-modules by $\op{mod}_G(\D_X)$ which is abelian and stable under taking subquotients within $\op{mod}(\D_X)$.


For an equivariant $\D$-module $M$ and a (locally) closed $G$-stable subset $Y\subset X$, all local cohomology modules $H^i_Y(M)$ are equivariant.


%

\subsection{Representation theory of the general linear group}\label{sec:repgl}

We recall some facts on the representation theory of $\GL_n(\C)$. We write $\bb{Z}^n_{dom}$ for the set of \defi{dominant weights} in $\bb{Z}^n$, i.e. tuples $\ll=(\ll_1,\cdots,\ll_n)\in\bb{Z}^n$ with $\ll_1\geq\ll_2\geq\cdots\geq\ll_n$. When each $\ll_i\geq 0$ we identify $\ll$ with a \defi{partition} with (at most) $n$ parts, and write $\ll\in\bb{N}^n_{dom}$. For a partition, we write $\lambda\vdash k$ when  $|\ll|:=\ll_1+\cdots+\ll_n=k$, in which case we can associate its corresponding \defi{Young diagram} with $k$ boxes that consists of $\lambda_i$ boxes in the $i$th row. The \defi{Durfee size} of $\ll$ is the largest $i$ with the property $\ll_i\geq i$.
We write $\lambda'$ for the \defi{conjugate} partition of $\lambda$, where $\ll'_i$ counts the number of parts $\ll_j$ with $\ll_j\geq i$. We partially order $\bb{Z}^n_{dom}$ (and $\bb{N}^n_{dom}$) by declaring $\ll\geq \mu$ if $\ll_i\geq\mu_i$ for all $i=1,\cdots,n$. 
If $a\geq 0$ then we write $a\times b$ or $(b^a)$ for the sequence $(b,b,\cdots,b)$ where $b$ is repeated $a$ times.
Given a weight $\ll\in\bb{Z}^n$ we write for its dual
\[ \ll^{\vee} = (-\ll_n,-\ll_{n-1},\cdots,-\ll_1).\]

If $V$ is a vector space with $\dim(V)=n$ and $\ll\in\bb{Z}^n_{dom}$ we write $\SS_{\ll}V$ for the corresponding irreducible representation of $\GL(V)$ (or \defi{Schur functor}). Our conventions are such that if $\ll=(k,0,\cdots,0)$ then $\SS_{\ll}V = \Sym^k V$, and if $\ll=(1^r)$ then $\bb{S}_{\ll}V=\bw^r V$.

%

For $a\geq b\geq 0$ we define the \defi{Gaussian (or $q$-)binomial coefficient} ${a\choose b}_q$ to be the polynomial in $\bb{Z}[q]$ defined by
\[{a\choose b}_q = \frac{(1-q^a)\cdot(1-q^{a-1})\cdots (1-q^{a-b+1})}{(1-q^b)\cdot(1-q^{b-1})\cdots (1-q)}.\]
One significance of the $q$-binomial coefficients is that ${a\choose b}_{q^2}$ describes the Poincar\'e polynomial of the Grassmannian $\op{Grass}(b,a)$ of $b$-dimensional subspaces of $\bb{C}^a$. As such, the coefficient of $q^j$ in ${a\choose b}_q$ computes the number of Schubert classes of (co)dimension $j$, or equivalently the number of partitions $\ll$ of size $j$ contained inside the rectangular partition $(a-b)\times b$. We get
\begin{equation}\label{eq:qbin-genfun}
{a\choose b}_q \, = \, \sum_{\ll\leq (b^{a-b})} q^{|\ll|}.
\end{equation}

\section{Singular cohomology of matrix orbits}
\label{sec:cohorb}

In this section, we compute the singular cohomology rings of the orbits $O_p$ of general, symmetric and skew-symmetric matrices. Throughout, we work with singular cohomology over complex coefficients. The computation of cohomology of homogeneous spaces is a well-studied problem in topology that generated (e.g., see \cites{cartan, borel, baum, may}) and continues to generate (e.g. \cite{franz}) a lot of interest. 

In order to determine the cohomology of the matrix orbits $O_p$, we use the classical method of H. Cartan 
\cite{cartan}. Let $K$ be a compact connected Lie group, and $L\subset K$ a closed connected Lie group. We have an induced map $\rho: \, H^*(BK) \to H^*(BL)$ between the cohomology rings of their classifying spaces. The following isomorphism of algebras reduces the problem at hand to an algebraic one (cf. \cite{cartan})
\begin{equation}\label{eq:classifying}
H^*(K/L) \, \cong \, \op{Tor}_{H^*(BG)}(\C, H^*(BL) ).
\end{equation}
We recall Cartan's result in a form that is most convenient for our calculations (see \cite[Theorem 8]{terzic}).

Let $T_L\subset T$ be an inclusion of corresponding maximal tori, and consider the complexification of their Lie algebras $\mf{t}_L \subset \mf{t}$. Denote the Weyl groups by $W(L)$ and $W$, which act naturally on the polynomial rings $\bb{C}[\mf{t}_L]$ and $\bb{C}[\mf{t}]$, respectively. We think of these rings having coordinate functions in degree two. The map $\rho$ takes the explicit form
\begin{equation}\label{eq:rho}
 \rho: \bb{C}[\mf{t}]^W \to \bb{C}[\mf{t}_L]^{W(L)}.
 \end{equation}
 
 Let $n=\op{rank} K$ and $r = \op{rank} L$. By a well-known theorem of Hopf (see \cite[Theorem 6.26]{mimu-toda}) the cohomology of $K$ is an exterior algebra
\[H^*(K) \cong \bw(z_1,\dots z_n),\]
where the generators $z_i$ have odd degrees. Our computations are based on the following version of (\ref{eq:classifying}).

\begin{theorem}\label{thm:cartan}
Let $f_1,\dots, f_n$ be homogeneous generators of the algebra $\bb{C}[\mf{t}]^W$ with $\deg f_i = \deg z_i+1$. If $\rho(f_{r+1}), \dots, \rho(f_{n})$ belong to the ideal $(\rho(f_1),\dots, \rho(f_r))$, then we have an isomorphism of graded algebras
\[ H^*(K/L) \cong \left(\bb{C}[\mf{t}_L]^{W(L)}/(\rho(f_1),\dots, \rho(f_r)) \right)\oo \bw (z_{r+1},\dots, z_n).\]
In particular, the Hilbert--Poincar\'e polynomial of $H^*(K/L)$ is given by 
\[\dfrac{(1-q^{\deg f_1})\cdots (1-q^{\deg f_r})}{(1-q^{d_1})\cdots (1-q^{d_r})} \cdot (1+q^{\deg f_{r+1}-1}) \cdots (1+q^{\deg f_n-1}),\]
where $d_1,\dots, d_r$ are the degrees of the fundamental invariants in the polynomial ring $\bb{C}[\mf{t}_L]^{W(L)}$.
\end{theorem}

Recall the facts about the cohomology of the Grassmannian $\op{Grass}(p,n)$ described in Section \ref{sec:repgl} (also, see (\ref{eq:grasspres}) in the proof below for an explicit presentation). We now proceed with determining the cohomology of the orbits. While for our subsequent applications we only use the Hilbert--Poincar\'e polynomials, for completeness we outline the argument yielding the explicit ring structure. In fact, for the symmetric case (c) this approach is necessary, since the generic stabilizers of $O_p$ ($p>0$) are disconnected. We note that in case (a), a description for $H^*(O_p)$ has been also obtained recently in \cite{zach}*{Proposition~3.6}. For the standard notions and identities involving symmetric functions, we refer the reader to \cite{mac}.

\begin{theorem}\label{thm:coho-orbit}
We have the following isomorphisms of graded algebras, and respective Hilbert--Poincar\'e polynomials (with $\deg z_i = 2i-1$):
\begin{itemize}
\item[(a)] When $X=\C^{m\times n}$:
\[H^*(O_p) \cong H^*(\op{Grass}(p,n)) \oo \bw (z_{m-p+1},\dots, z_{m}), \]
\[{n \choose p}_{q^2} \cdot (1+q^{2m-2p+1})\cdot (1+q^{2m -2p+3}) \cdots (1+q^{2m-1}).\]
In particular, we have that
\[b^{tot}(O_p) = {n\choose p}\cdot 2^p.\]
\item[(b)] When $X=\bw^2 \C^n$, with $m=\lfloor n/2\rfloor$, and $\epsilon = n-2m$, then
\[H^*(O_p) \cong \C[h_1,\dots,h_p]/(h_{m-p+1}, \dots, h_{m}) \oo \bw (z_{n+\epsilon-2p+1},z_{n+\epsilon-2p+3}, \dots, z_{n+\epsilon -1}),\]
\[{m \choose p}_{q^4} \cdot (1+q^{2(n+\epsilon)-4p+1})\cdot (1+q^{2(n+\epsilon)-4p+5}) \cdots (1+q^{2(n+\epsilon)-3}).\]
Here $h_i$ stands for the $i$th complete homogeneous symmetric polynomial in $p$ variables and $\deg h_i = 4i$. In particular, we have that
\[b^{tot}(O_p) = {m\choose p}\cdot 2^p.\]

\item[(c)] Suppose that $X= \Sym^2 \C^n$, let $m=\lfloor n/2\rfloor$, and let $\epsilon = n-2m$. If $p=2r$ then
\[H^*(O_p) \cong \C[h_1,\dots,h_r]/(h_{m-r+1}, \dots, h_{m}) \oo \bw (z_{n+\epsilon-p+1},z_{n+\epsilon-p+3}, \dots, z_{n+\epsilon -1}),\]
\[{m \choose r}_{q^4} \cdot (1+q^{2(n+\epsilon)-4r+1})\cdot (1+q^{2(n+\epsilon)-4r+5}) \cdots (1+q^{2(n+\epsilon)-3}),\]
and in particular we have that
\[b^{tot}(O_p) = {m\choose r} \cdot 2^r.\]
If $p=2r+1$ then 
\[H^*(O_p) \cong \C[h_1,\dots,h_r]/(h_{m-r+\epsilon}, \dots, h_{m-1+\epsilon}) \oo \bw (z_{2m-2r+1},z_{2m-2r+3}, \dots, z_{2m-1}, z_{n}),\]
\[{m-1+\epsilon \choose r}_{q^4} \cdot [(1+q^{4m-4r+1})\cdot (1+q^{4m -4r+5}) \cdots (1+q^{4m-3})] \cdot (1+q^{2n-1}),\]
and in particular we have that
\[b^{tot}(O_p) = {m\choose r} \cdot 2^{r+1}.\]
Here  $h_i$ stands for the $i$th complete homogeneous symmetric polynomial in $r$ variables and $\deg h_i = 4i$. 
\end{itemize}
\end{theorem}

\begin{proof}
We consider first part (a).  It is easy to see that we have
\[O_p \cong (\GL_m(\bb{C}) \times \GL_n(\bb{C})) / H,\]
where $H$ denotes the stabilizer of $\begin{bmatrix} I_p & 0 \\ 0 & 0 \end{bmatrix}$, equal to the subgroup of pairs of matrices of the form 
\[\left(\begin{bmatrix} A & B \\ 0 & C \end{bmatrix},\begin{bmatrix} A & 0\\ D & E \end{bmatrix}\right), \;\mbox{ with } A\in \GL_p(\bb{C}), B \in \bb{C}^{p\times (m-p)}, C\in \GL_{m-p}(\bb{C}), D \in \bb{C}^{(n-p) \times p}, E\in \GL_{n-p}(\bb{C}). \]
Since the product of unitary groups $U(m) \times U(n)$ is the maximal compact subgroup of $\GL_m(\bb{C}) \times \GL_n(\bb{C})$, we have by \cite[Theorem 3.1]{mostow} that $O_p$ has the same homotopy type as $K/L$, where
\[K =U(m) \times U(n), \, \mbox{ and } \, L= U(p)\times U(m-p) \times U(n-p).\]
The rings of invariants from (\ref{eq:rho}) are polynomial rings, generated by elementary symmetric polynomials:
\[\bb{C}[\mf{t}]^W  = \bb{C}[x_1,\dots,x_m, y_1,\dots, y_n], \quad \mbox{ and } \quad \bb{C}[\mf{t}_L]^{W(L)} = \bb{C}[a_1,\dots,a_p, b_1, \dots , b_{m-p}, c_1, \dots, c_{n-p}],\]
where $\deg x_i = \deg y_i = \deg a_i = \deg b_i = \deg c_i = 2i$ (below we allow $i=0$). The map $\rho$ from (\ref{eq:rho}) is given by (compare \cite[Theorem 5.8]{mimu-toda})
\[ \rho(x_k) = \sum_{i+j=k} a_i b_j, \qquad \rho(y_k)=\sum_{i+j=k} a_i c_j. \]
Let $I$ denote the ideal generated by all the elements $\rho(x_i)$ ($1\leq i \leq m$) and $\rho(y_j)$ ($1\leq j\leq n$). Using successively that $\rho(x_k)-\rho(y_k) \in I$ for $k=1,\dots, m-p$, we see that $b_k-c_k \in I$ ($1\leq k \leq n-p$) and $b_j \in I$ ($n-p+1\leq j \leq m-p$). Therefore, we have 
\[I=\big(\rho(x_1),\dots,\rho(x_{m-p}), \rho(y_1),\dots, \rho(y_{n})\big).\]
Hence, by Theorem \ref{thm:cartan} we obtain 
\[H^*(K/L) \cong  R \oo \bw (z_{m-p+1},\dots, z_m),\]
with 
\begin{equation}\label{eq:grasspres}
R=\bb{C}[\mf{t}_L]^{W(L)}/I\cong \bb{C}[a_1,\dots,a_p, c_1, \dots , c_{n-p}]/(\rho(y_1),\dots, \rho(y_n)),
\end{equation}
which is a well-known presentation of $H^*(\op{Grass}(p,n))$ (see \cite[Theorem 6.9]{mimu-toda}).

Now we turn to part (b). By working with the representative 
$$\begin{bmatrix} 0 & I_p & 0\\ -I_p & 0 & 0 \\ 0 & 0 & 0 \end{bmatrix} \in O_p,$$ 
we see as above by \cite[Theorem 3.1]{mostow} that $O_p$ has the same homotopy type as $K/L$, with $K=U(n),$ and $L=\op{Sp}(p) \times U(n-2p)$, where $\op{Sp}(p)=\op{Sp}(2p,\C) \, \cap \, U(2p)$ is the compact symplectic Lie group.

We let the Cartan subalgebra $\mf{t}$ be the set of diagonal matrices 
$\op{diag}(a_1,\dots,a_p, a_{p+1},\dots, a_{2p}, b_1,\dots, b_{n-2p})$ 
(where $a_i, b_i \in \bb{C}$), while the Cartan subalgebra $\mf{t}_L \subset \mf{t}$ to be $\op{diag}(a_1,\dots,a_p, -a_1,\dots -a_p, b_1,\dots, b_{n-2p})$. The corresponding Weyl groups are $W=S_n$ and $W_L=(S_p\ltimes \bb{Z}_2^p) \times S_{n-2p}$, acting in the obvious way -- the symmetric group by permutations, and $\bb{Z}_2 = \bb{Z}/2\bb{Z}$ by sign changes. 
Let $x_i, y_i$ be the coordinate functions corresponding to $a_i, b_i$, respectively, and for $k \in \bb{Z}_{>0}$ consider the power sum polynomials 
\[p_k = \sum_{i=1}^{2p} x_i^k + \sum_{j=1}^{n-2p} y_j^k, \qquad q_{k} = \sum_{i=1}^p x_i^k, \qquad r_{k}= \sum_{i=1}^{n-2p} y_i^k.\]
The respective rings of invariants are polynomial rings generated by
\[\bb{C}[\mf{t}]^W  = \bb{C}[p_1,\dots, p_n], \quad \mbox{ and } \quad \bb{C}[\mf{t}_L]^{W(L)} = \bb{C}[q_2, q_4, \dots, q_{2p}, r_{1}, r_{2},\dots, r_{n-2p}].\]
The map $\rho$ from (\ref{eq:rho}) is given by (where $1\leq k\leq n$)
\[ \rho(p_k) = 2q_k + r_k, \mbox{\, for } k \mbox{ even, \quad and \quad} \rho(p_k)=r_k, \mbox{ for } k \mbox{ odd}.\]
Since $\bb{C}[y_1,\dots, y_{n-2p}]^{S_{n-2p}} = \bb{C}[r_1,\dots,r_{n-2p}]$, we see that $\rho(p_k) \in (r_1,r_3,\dots, r_{n-2p-1+\epsilon})$ for all odd $k$. By Theorem \ref{thm:cartan}, we obtain
\[H^*(O_p) \cong \left( \bb{C}[q_2,\dots, q_{2p}, r_1, \dots, r_{n-2p}]/I\right) \oo \bigwedge (z_{n+\epsilon-2p+1},z_{n+\epsilon-2p+3},\dots, z_{n+\epsilon-1}),\]
where $I=(r_1,r_3,\dots, r_{n-2p-1+\epsilon}, 2q_2 + r_2, 2q_4 +r_4, \dots, 2q_{n-\epsilon}+r_{n-\epsilon})$ and $\deg z_i = 2i-1$. For $k\in \bb{N}$, let $e_k$ (resp. $h_k$) denote the $k$th elementary (resp. complete) symmetric polynomial in the variables $y_1,\dots,y_{n-2k}$ (resp. in $x_1^2,\dots, x_p^2$), so that if $k>n-2p$ then $e_k = 0$. We claim that for all $0\leq k \leq m$ we have
\begin{equation}\label{eq:elemcomp}
e_{2k} - h_k \in I, \quad \mbox{ and } \quad e_{2k+1} \in I.
\end{equation}
The latter part follows readily by induction using the Girard--Newton identities and the fact that $r_i \in I$ when $1\leq i\leq n-2p$ is odd.

Now we prove that $e_{2k} - h_k \in I$, again by induction, the case $k=0$ being trivial. We have the following equalities modulo $I$, again using the Girard--Newton identities and that $e_{2i+1} \in I$:
\[2k \cdot e_{2k} \equiv \sum_{i=1}^{2k} (-1)^{i-1} e_{2k-i} \cdot r_i \equiv  \sum_{i=1}^k - h_{k-i} \cdot (-2 q_{2i}) \equiv 2k \cdot h_k.\]
This proves the first claim in (\ref{eq:elemcomp}) as well, which now implies part (b) since
\[\bb{C}[q_2,\dots, q_{2p}, r_1, \dots, r_{n-2p}]/I = \bb{C}[h_1,\dots, h_p, e_1, \dots, e_{n-2p}]/I\cong \bb{C}[h_1,\dots, h_p]/(h_{m-p+1},\dots, h_m).\]

Now consider part (c). By choosing the representative $$\begin{bmatrix} I_{p} & 0 \\ 0 & 0 \end{bmatrix} \in O_p,$$ 
we see as before that that $O_p$ has the same homotopy type as $K/L'$, with
\[K=U(n), \, \mbox{ and } \, L'=O(p,\bb{R}) \times U(n-p).\]
We first use Theorem \ref{thm:cartan} in order to compute the cohomology ring of $K/L$, where $L=L^{\prime 0} =SO(p,\bb{R}) \times U(n-p)$ is the connected component of $L'$ containing the identity. 

Assume first that $p$ is even. For $a,a'\in \bb{C}$, denote by $R(a,a')$ the $2\times 2$ matrix $\,1/2 \cdot \begin{bmatrix} (a+a') & a-a' \\ a'-a & (a+a')\end{bmatrix}$. 
Let $\mf{t}$ be the Cartan subalgebra $\op{diag}(R(a_1,a_{r+1}), R(a_2,a_{r+2}), \dots, R(a_r,a_{2r}), b_1,\dots, b_{n-p})$ formed of block diagonal matrices, where $a_i,b_i \in \bb{C}$. The Weyl group $W=S_n$ acts by permuting the entries $a_1,\dots, a_{p}, b_1,\dots, b_{n-p}$ in the usual way. We choose $\mf{t}_L \subset \mf{t}$ to be $\op{diag}(R(a_1,-a_1), R(a_2,-a_2), \dots, R(a_r, -a_{r}), b_1,\dots, b_{n-p})$. Here the first factor of the Weyl group $W_L = (S_r\ltimes \bb{Z}_2^{r-1}) \times S_{n-p}$
acts on $a_1,\dots, a_r$ by permutations and an even number of sign changes. Let $x_i, y_i$ be the coordinate functions corresponding to $a_i, b_i$, respectively, and for $k \in \bb{Z}_{>0}$ consider polynomials 
\[p_k = \sum_{i=1}^{p} x_i^k + \sum_{j=1}^{n-p} y_j^k, \qquad q_{k} = \sum_{i=1}^r x_i^k,  \qquad q=\prod_{i=1}^{r} x_i, \qquad r_{k}= \sum_{i=1}^{n-p} y_i^k.\]
The respective rings of invariants are polynomial rings generated by
\[\bb{C}[\mf{t}]^W  = \bb{C}[p_1,\dots, p_n], \quad \mbox{ and } \quad \bb{C}[\mf{t}_L]^{W(L)} = \bb{C}[q_2, q_4, \dots, q_{p-2}, q, r_{1}, r_{2},\dots, r_{n-p}].\]
The map $\rho$ from (\ref{eq:rho}) is given by (where $1\leq k\leq n$)
\[ \rho(p_k) = 2q_k + r_k, \mbox{\, for } k \mbox{ even, \quad and \quad} \rho(p_k)=r_k, \mbox{ for } k \mbox{ odd}.\]
As in case (b), we obtain
\[H^*(K/L) \cong \left( \bb{C}[q_2,\dots, q_{p-2},q, r_1, \dots, r_{n-p}]/I\right) \oo \bigwedge (z_{n+\epsilon-p+1},z_{n+\epsilon-p+3},\dots, z_{n+\epsilon-1}),\]
where $I=(r_1,r_3,\dots, r_{n-p-1+\epsilon}, 2q_2 + r_2, 2q_4 +r_4, \dots, 2q_{n-\epsilon}+r_{n-\epsilon})$ and $\deg z_i = 2i-1$. Now the action of $-1\in\bb{Z}_2 \cong O(p,\bb{R})/SO(p,\bb{R})$ leaves $q_{2k}$ (and
$z_i,r_j$) invariant, but sends $q$ to $-q$. Hence, we have
\[H^*(O_p) \cong H^*(K/L)^{\bb{Z}_2} \cong \left( \bb{C}[q_2,\dots, q_{p-2},q_{p}, r_1, \dots, r_{n-p}]/I'\right) \oo \bigwedge (z_{n+\epsilon-p+1},z_{n+\epsilon-p+3},\dots, z_{n+\epsilon-1}),\]
with $I'$ having the same generators as those given for $I$. The rest of the proof follows as (\ref{eq:elemcomp}) in case (b).

Lastly, we consider case (c) with $p$ odd. We use similar notation as in the even case. Choose $\mf{t}$ to be $\op{diag}(R(a_1,a_{r+1}), R(a_2,a_{r+2}), \dots, R(a_r,a_{2r}), b_0, b_1,\dots, b_{n-p})$. Then $W=S_n$ acts by permuting the entries $a_1,\dots, a_{2r},b_0, b_1,\dots, b_{n-p}$. Choose $\mf{t}_L \subset \mf{t}$ to be $\op{diag}(R(a_1,-a_1), R(a_2,-a_2), \dots, R(a_r, -a_{r}), 0, b_1,\dots, b_{n-p})$. The first factor of $W_L = (S_r\ltimes \bb{Z}_2^{r}) \times S_{n-p}$
acts on $a_1,\dots, a_r$ by permutations and sign changes. Consider 
\[p_k = \sum_{i=1}^{2r} x_i^k + \sum_{j=0}^{n-p} y_j^k, \qquad q_{k} = \sum_{i=1}^r x_i^k,  \qquad r_{k}= \sum_{i=1}^{n-p} y_i^k.\]
The rings of invariants are
\[\bb{C}[\mf{t}]^W  = \bb{C}[p_1,\dots, p_n], \quad \mbox{ and } \quad \bb{C}[\mf{t}_L]^{W(L)} = \bb{C}[q_2, q_4, \dots, q_{2r}, r_{1}, r_{2},\dots, r_{n-p}].\]
The map $\rho$ from (\ref{eq:rho}) is given by (where $1\leq k\leq n$)
\[ \rho(p_k) = 2q_k + r_k, \mbox{\, for } k \mbox{ even, \quad and \quad} \rho(p_k)=r_k, \mbox{ for } k \mbox{ odd}.\]
As in case (b), we obtain
\[H^*(K/L) \cong \left( \bb{C}[q_2,\dots, q_{2r}, r_1, \dots, r_{n-p}]/I\right) \oo (z_{2m-2r+1},z_{2m-2r+3}, \dots, z_{2m-1}, z_{n}),\]
where $I=(r_1,r_3,\dots, r_{n-p-\epsilon}, 2q_2 + r_2, 2q_4 +r_4, \dots, 2q_{n-2+\epsilon}+r_{n-2+\epsilon})$ and $\deg z_i = 2i-1$ (here we used also the fact that $\rho(p_n)\in I$ since $y_0 \mapsto 0$). Now the action of $-1\in\bb{Z}_2 \cong O(p,\bb{R})/SO(p,\bb{R})$ leaves all $q_{2k}, z_i, r_j$ invariant. Thus, $H^*(O_p) \cong H^*(K/L)^{\bb{Z}_2} = H^*(K/L)$, and the rest of the proof follows again as in case (b).
\end{proof}

\section{The case of $m\times n$ matrices} \label{sec:matrices}

In this section we let $X=\bb{C}^{m\times n}$ denote the space of $m\times n$ complex matrices, endowed with the natural action of $G=\GL_m\times\GL_n$ via row and column operations.  The coordinate ring $S$ of $X$ can be identified with the polynomial ring $S = \bb{C}[x_{ij}]$, where $1\leq i\leq m$ and $1\leq j\leq n$. We assume that $m\geq n$, so the orbits of this action are the sets $O_p$ consisting of matrices of rank $p$, for $p=0,\cdots,n$, and their closures are given by
\[ \ol{O}_p = \bigcup_{i=0}^p O_i.\] 
The goal of this section is to prove the following result, which combined with \eqref{eq:qbin-pos} implies part a) of Theorems~\ref{thm:main2},~\ref{thm:di=0}, and~\ref{thm:main}.

\begin{theorem}\label{thm:main-generic}
Suppose that $0\leq p<n\leq m$.
\begin{itemize} 
\item[(a)] The generating function for de Rham cohomology of local cohomology modules is
\[ \sum_{i,j\geq 0} h^i_{dR}(H^j_{\ol{O}_p}(S))\cdot q^i \cdot w^j = \sum_{s=0}^p q^{(m-s)\cdot(n-s)} \cdot {n \choose s}_{q^2} \cdot w^{(n-p)^2+(n-s)\cdot(m-n)} \cdot {n-1-s \choose p-s}_{w^2}.\]

\item[(b)] The Hilbert--Poincar\'e polynomial for the Borel--Moore homology of the orbit closures is given by
\[\sum_{i\geq 0} h_i^{BM}(\ol{O}_p) \cdot q^i \, = \, \sum_{s=0}^p q^{2sm + (p-s)(p-s+2)} \cdot {n \choose s}_{q^{2}} \cdot {n-1-s \choose p-s}_{q^{2}}.\]

\item[(c)] The \v Cech--de Rham spectral sequence \eqref{eq:CdR-sp-seq} degenerates at the $E_2$ page, and the maps $d_i$ in \eqref{eq:les-BM-hom} vanish.
\end{itemize}
\end{theorem}

The restriction $p<n$ in Theorem~\ref{thm:main-generic} is made in order to avoid the trivial case $p=n$ when $\ol{O}_n=X$ (see \eqref{eq:BM-X} and \eqref{eq:CdR-for-X}). To prove Theorem~\ref{thm:main-generic} we follow closely the outline described in the Introduction, and explain the details in Section~\ref{subsec:proof-generic}. We then consider in Section~\ref{subsec:Lyub-comparison} some further consequences of Theorem~\ref{thm:main-generic} and discuss the relationship with Lyubeznik numbers.

\subsection{The proof of Theorem~\ref{thm:main-generic}}\label{subsec:proof-generic}

The simple objects in $\opmod_{G}(\D_X)$ are $D_0,\cdots,D_n$, where $D_p=\mc{L}(\ol{O}_p, X)$ denotes the intersection homology $\D_X$-module corresponding to the trivial local system on the orbit $O_p$ (see \cite{raicu-dmods}*{Theorem~2.9}). By \cite{euler}*{Theorem 4.1} (see also \cite{zelevinsky}*{Section~3.3}), the generating function for the de Rham cohomology of the simples $D_s$ is given by

\begin{equation}\label{eq:deRham-Dp}
\sum_{i\geq 0} h^i_{dR}(D_s) \cdot q^i =  {n \choose s}_{q^2} \cdot q^{(m-s)\cdot(n-s)}.
\end{equation}
Moreover, by \cite{raicu-weyman-loccoh}*{(1.3)}, we have for $p<n$ the formal identity
\begin{equation}\label{eq:compos-loccoh-generic}
\sum_{j\geq 0} [H^j_{\ol{O}_p}(S)]\cdot w^j = \sum_{s=0}^p [D_s]\cdot w^{(n-p)^2+(n-s)\cdot(m-n)} \cdot {n-1-s \choose p-s}_{w^2}.
\end{equation}
describing the simple $\mc{D}_X$-composition factors of the local cohomology modules $H^j_{\ol{O}_p}(S)$. Combining \eqref{eq:deRham-Dp} with \eqref{eq:compos-loccoh-generic}, and using the fact that de Rham cohomology is subadditive in short exact sequences, we obtain the inequality
\begin{equation}\label{eq:ineq-gen-fun-DR-loccoh-generic} \sum_{i,j\geq 0} h^i_{dR}(H^j_{\ol{O}_p}(S))\cdot q^i \cdot w^j \leq \sum_{s=0}^p q^{(m-s)\cdot(n-s)} \cdot {n \choose s}_{q^2} \cdot w^{(n-p)^2+(n-s)\cdot(m-n)} \cdot {n-1-s \choose p-s}_{w^2}.
\end{equation}

\begin{remark}\label{rem:semisimple}
 In the case when $m>n$, the category $\opmod_{G}(\D_X)$ is semi-simple by \cite[Theorem 5.4]{lor-wal}, hence \eqref{eq:compos-loccoh-generic} encodes a direct sum decomposition of local cohomology modules into a sum of simples. Taking de Rham cohomology is therefore additive, and we get that \eqref{eq:ineq-gen-fun-DR-loccoh-generic} is an equality. This argument however fails in the case $m=n$ when the groups $H^j_{\ol{O}_p}(S)$ are no longer direct sums of simple modules (see Section~\ref{subsec:Lyub-comparison} below).
\end{remark}

We define $N_p$, $p<n$, to be the specialization of the right side of \eqref{eq:ineq-gen-fun-DR-loccoh-generic} to $q=w=1$, namely
\begin{equation}\label{eq:Np-generic}
    N_p = \sum_{s=0}^p {n\choose s}\cdot {n-1-s\choose p-s},
\end{equation}
and observe that specializing the left side of \eqref{eq:ineq-gen-fun-DR-loccoh-generic} to $q=w=1$ we get
\[ \rho^{tot}(\ol{O}_p) \leq N_p.\]

\begin{lemma}\label{lem:btot-Op-Np-generic}
 We have for $p<n$ that \eqref{eq:btotOp=NpNpm1} holds.
\end{lemma}

\begin{proof}
 Since $p<n$, we have that 
 \[
 \begin{aligned}
 N_p+N_{p-1} &= \sum_{s=0}^p {n\choose s}\cdot {n-1-s\choose p-s} + \sum_{s=0}^{p-1} {n\choose s}\cdot {n-1-s \choose p-1-s} \\
 &=\sum_{s=0}^p {n\choose s}\cdot {n-s\choose p-s} = \sum_{s=0}^p {n\choose p}\cdot{p\choose s} = {n\choose p}\cdot 2^p.
 \end{aligned}\]
 The desired conclusion now follows from Theorem~\ref{thm:coho-orbit}(a).
\end{proof}

As explained in the Introduction, the equality \eqref{eq:btotOp=NpNpm1} implies the degeneration of the spectral sequence \eqref{eq:CdR-sp-seq} (for both $\ol{O}_p$ and $\ol{O}_{p-1}$), and the vanishing of the maps $d_i$ in \eqref{eq:les-BM-hom}, hence Theorem~\ref{thm:main-generic}(c) holds. Moreover, \eqref{eq:btotOp=NpNpm1} also implies that \eqref{eq:ineq-gen-fun-DR-loccoh-generic} is an equality, proving Theorem~\ref{thm:main-generic}(a). The degeneration of \eqref{eq:CdR-sp-seq}, together with the fact that $d_X=mn$, implies that
\[\sum_{k\geq 0} h_k^{BM}(\ol{O}_p) \cdot q^{2mn-k} = \sum_{i,j\geq 0} h^i_{dR}(H^j_{\ol{O}_p}(S))\cdot q^{i+j}\]
is obtained by specializing the equality in part (a) to $w=q$. Making the change of variable $q\to q^{-1}$ and multiplying by $q^{2mn}$ we get
\[\sum_{k\geq 0} h_k^{BM}(\ol{O}_p) \cdot q^{k} = \sum_{s=0}^p q^{2mn-(m-s)\cdot(n-s)-(n-p)^2-(n-s)\cdot(m-n)} \cdot {n \choose s}_{q^{-2}} \cdot {n-1-s \choose p-s}_{q^{-2}},\]
and Theorem~\ref{thm:main-generic}(b) now follows using the identity (\ref{eq:qbin-pos}).

\subsection{Comparison with Lyubeznik numbers}\label{subsec:Lyub-comparison}
As explained in Remark \ref{rem:semisimple}, when $m\neq n$, the category $\op{mod}_G(\D_X)$ is semisimple, yielding to a simpler argument for obtaining the \v Cech--de Rham numbers. Since $\mc{F}(D_p)\cong D_{n-p}$ (e.g. see \cite[Remark 1.5]{raicu-dmods}), by (\ref{eq:pushpull}) the \v Cech--de Rham numbers are determined completely by the Lyubeznik numbers, and vice-versa (up to relabeling). 

From now on we assume that $m=n$, when the situation is more interesting since $\op{mod}_G(\D_X)$ is no longer semisimple. Nevertheless, when $p<n$ the $\D$-module $H^j_{\ol{O}_p}(S)$ can be written as a direct sum of the indecomposable $\D$-modules $Q_0, Q_1, \dots, Q_p$ \cite[Theorem 1.6]{lor-rai}, with
\[Q_n= S_{\det}, \qquad Q_p = \dfrac{S_{\det}}{\langle \det^{p-n+1} \rangle_{\D}} \quad (0\leq p \leq n-1),\]
where $S_{\det}$ denotes the localization of $S$ at the determinant, and $\langle \det^{p-n+1} \rangle_{\D}$ is the $\D$-submodule generated by ${\det}^{p-n+1}$. Note that $Q_0=D_0$ and for $1 \leq p \leq n$, we have the short exact sequences (cf. \cite{lor-rai})
\begin{equation}\label{eq:ses-Qp}
 0 \lra D_p \lra Q_p \lra Q_{p-1} \lra 0.
\end{equation}
The short exact sequence (\ref{eq:ses-Qp}) is not split in the category of $\D$-modules, but it is split in the category of rational $G$-representations. We obtain a decomposition of $Q_p$ as a $G$-representation
\[ Q_p = \bigoplus_{s=0}^p Q_p^s,\]
where $Q_p^s \simeq D_s$. As a rational $G$-representation, the decomposition of $D_p$ is given in \cite[Theorem~6.1]{raicu-weyman}, \cite[Main Theorem(1)]{raicu-weyman-loccoh}, or \cite[Theorem~5.1]{raicu-survey}. We fix our conventions as follows. Let $V_1,V_2$ be vector spaces, $\dim(V_i)=n$, let $S = \Sym(V_1\oo V_2)$, and identify $X=\Spec(S) = V_1^{\vee} \oo V_2^{\vee}\cong \bb{C}^{m\times n}$ with the action of the group $G= \GL(V_1) \times \GL(V_2)$ as before. Then
\begin{equation}\label{eq:decomp-Dp}
 D_p = \bigoplus_{\ll \in W(p)}\bb{S}_{\ll} V_1 \oo \bb{S}_{\ll} V_2,
\end{equation}
where 
\[ W(p) \, = \, \{\ll\in\bb{Z}^n_{\dom} : \ll_p\geq p-n\geq \ll_{p+1}\}.\]

\begin{lemma}\label{lem:Qps}
 If $\pd\in V_1^{\vee}\oo V_2^{\vee}$ is a derivation and $z\in Q_p^s$ then $\pd(z)\in Q_p^s$. 
\end{lemma}

\begin{proof}
 Without loss of generality, we may assume that $z$ belongs to an isotypic component $\bb{S}_{\ll}V_1 \oo \bb{S}_{\ll}V_2$, where $\ll\in W(s)$. Since $\pd\in V_1^{\vee}\oo V_2^{\vee}$, it follows from Pieri's rule \cite[Corollary 2.3.5]{weyman} and the fact that $Q_p$ is closed under the action of $\pd$, that
 \[ \pd(z) = \sum_{\nu} z_{\nu},\]
 where for each $\nu$ we have that $z_{\nu}\in \bb{S}_{\nu}V_1 \oo \bb{S}_{\nu}V_2$ for $\nu\in W(t)$ with $t\leq p$, and $\nu$ is obtained from $\ll$ by removing one box, i.e. there exists an index $1\leq r\leq n$ such that
 \[ \nu_i = \ll_i\mbox{ for }i\neq r,\mbox{ and } \nu_r=\ll_r-1.\]
 Since $z\in\ker(Q_p \onto Q_{s-1})$ it follows that $s\leq t\leq p$. There are two possibilities:
 \begin{itemize}
  \item $r\neq s$. Then $\nu_s=\ll_s\geq s-n$, and $\nu_{s+1}\leq\ll_{s+1}\leq s-n$. So $\nu\in W(s)$, which means $z_{\nu}\in Q_p^s$.
  \item $r=s$. If $\ll_s>s-n$ then $\nu_s\geq s-n$ and $z_{\nu}\in Q_p^s$. If $\ll_s=s-n$ then $\nu_s=s-n-1$, so $\nu\not\in W(s)$. It follows that $\nu\in W(t)$ for some $t>s$. However, this implies that
  \[ t-n\leq \nu_t \leq \nu_s = s-n-1,\]
  which yields $t < s$, a contradiction. \qedhere
 \end{itemize}
\end{proof}

In other words, the result above shows that the non-split exact sequence of $\D$-modules (\ref{eq:ses-Qp}) splits as $\C[\partial_{ij}]$-modules. Since the differentials (\ref{eq:diff-derham}) in the de Rham complex use only the $\C[\partial_{ij}]$-module structure of a $\D$-module, the following is an immediate consequence.

\begin{corollary}\label{cor:drQp}
We have a decomposition of complexes 
\[DR(Q_p) = \bigoplus_{s=0}^p DR(D_s).\]
In particular, 
 \[H^i_{dR}(Q_p) = \bigoplus_{s=0}^p H^i_{dR}(D_s).\]
\end{corollary}

\begin{remark}\label{rem:topo}
In the case $p=n$ the de Rham cohomology of $Q_n=S_{\det}$ coincides with the singular cohomology of the 
complement $O_n$ of the hypersurface $\det=0$, and that of $D_s$ yields intersection cohomology. Reinterpreting the result above in topological terms yields the following formula (here $c_s=\codim_{X} O_s$):
\[H^i(O_n,\C) = \sum_{s=0}^n IH^{i-c_s}(\ol{O}_s), \quad \mbox{ for all } i\geq 0.\]
\end{remark}

\begin{remark}\label{rem:dual} 
We end this subsection by concluding that the \v Cech--de Rham numbers only depend on the class of the local cohomology modules in the Grothendieck group of $\op{mod}_G(\D_X)$, whose description is uniform for the square and non-square cases. This is in contrast to the case of Lyubeznik numbers, where the formulas in the square case are different from the ones in the non-square case (see \cite[Theorems~1.3 and~1.5]{lor-rai}). The explanation in the case of Lyubeznik numbers comes from the fact that the sequence (\ref{eq:ses-Qp}) is not split in the category of $S$-modules. However, the sequence is split in the category of $\C[\pd_{ij}]$-modules, which is why the results on de Rham cohomology are uniform.
\end{remark}

\subsection{Mixed Hodge structure on cohomology and Borel--Moore homology}\label{sec:mhs-det}

In this section we compute the Hodge numbers of $H^i(O_p)$ and $H_i^{BM}(\ol{O}_p)$. This is based on the knowledge of the weight filtration on $H^i_{\ol{O}_p}(\mc{O}^H_X)$ by \cite[Theorem 1.1]{perlman2}, and the degeneration of the spectral sequence in Proposition \ref{prop:mhs} by Theorem \ref{thm:main-generic} (c). We first record the following result on intersection cohomology.

\begin{lemma}\label{lem:interhom}
For all $i,p \geq 0$, we have an isomorphism as mixed Hodge structures
\[IH^i(\ol{O}_p) \cong H^i(\op{Grass}(p,n)).\]
In particular, $IH^i(\ol{O}_p)$ has a pure Hodge structure of weight $i$, and its Hodge numbers are concentrated on the diagonal.
\end{lemma}

\begin{proof}
By \cite[Section 3.3]{zelevinsky}, there is a small resolution of singularities $Z \to \ol{O}_p$, such that $Z$ is the total space of a vector bundle over $\op{Grass}(p,n)$. This implies that we have an isomorphism of mixed Hodge structures $ IH^i(\ol{O}_p) \cong H^i(Z)$, for every $i\geq 0$ (see \cite[Proposition 8.2.30]{htt}). Since the Serre spectral sequence corresponding to the fibration $\pi: Z \to \op{Grass}(p,n)$ degenerates, the pullback via $\pi$ induces isomorphisms $H^i(Z) \cong H^i(\op{Grass}(r,n))$ of mixed Hodge structures, for all $i$, thus proving the first claim. As $\op{Grass}(p,n)$ is a smooth projective variety, this shows that $IH^i(\ol{O}_p)$ has a pure Hodge structure of weight $i$. The claim regarding the Hodge numbers follows from \cite[Example 19.1.11]{fulton}.
\end{proof}

Next, we record the Hodge numbers on de Rham cohomology of local cohomology.

\begin{theorem}\label{thm:derham-local-mhs}
The following trivariate generating function records the weight filtration on the mixed Hodge structure of $H^i_{dR} (H^j_{\ol{O}_p}(\mc{O}_X^H))$:
\[ \sum_{i,j,t\geq 0} \!\! \dim_{\C} \op{Gr}_k^{W} \! H^i_{dR}(H^j_{\ol{O}_p}(\mc{O}_X^H))\cdot q^i w^j t^k = \sum_{s=0}^p t^{p-s} \cdot (qt)^{(m-s)\cdot(n-s)} \cdot {n \choose s}_{\!(qt)^2} \!\!\!\!\!\! \cdot (wt)^{(n-p)^2+(n-s)\cdot(m-n)} \cdot {n-1-s \choose p-s}_{\!(wt)^2}\!\!\!\!.\]
Moreover, the Hodge numbers of $H^i_{dR} (H^j_{\ol{O}_p}(\mc{O}_X^H))$ are concentrated on the diagonal for all $i,j\geq 0$.
\end{theorem}

\begin{proof}
We write $D_p^H$ for the pure Hodge module of weight $d_{O_p}$ corresponding to the intersection cohomology sheaf of $\ol{O}_p$, isomorphic to $\mc{L}(\ol{O}_p, X) = D_p$ as $\D$-modules. By the discussion in Sections \ref{subsec:derham} and \ref{sec:mhs}, we have isomorphisms of mixed Hodge structures 
\begin{equation}\label{eq:drint}
H^i_{dR}(D_p^H) \cong IH^{i-c_p}(\ol{O}_p).    
\end{equation}

As seen in Section \ref{subsec:proof-generic}, the de Rham cohomology of $H^j_{\ol{O}_p}(S)$ is equal to the direct sum of the de Rham cohomology of its $\D$-module composition factors. By \cite[Theorem 3.1]{perlman2}, each of these factors $D_s$ carries the mixed Hodge module structure $D_s^H(d_{O_s}+s-mn-p-j)$, as a factor of $H^j_{\ol{O}_p}(\mc{O}_X^H)$. Due to the additivity of Hodge numbers, by (\ref{eq:drint}) each factor $D_s$ thus contributes with the Hodge numbers of $IH^{i-c_s}(\ol{O}_p)((d_{O_S}+s-mn-p-j)/2)$. In particular, by Lemma \ref{lem:interhom} all of these are concentrated on the diagonal, and the contribution of the de Rham cohomology of a factor $D_s$ to $H^i_{dR}(H^j_{\ol{O}_p}(\mc{O}_X^H))$ can occur only in weight $p-s+i+j$. Taking these into account, the combination of (\ref{eq:deRham-Dp}) and (\ref{eq:compos-loccoh-generic}) readily gives the desired formula.
\end{proof}

\noindent\textit{Proof of Theorem \ref{thm:main-weight}}.
By Theorem \ref{thm:main-generic} (c), we know that the spectral sequence of mixed Hodge modules in Proposition \ref{prop:mhs} degenerates at the $E_2$ page. As we did at the end of Section \ref{subsec:proof-generic}, we readily recover the first formula in Theorem \ref{thm:main-weight} from Theorem \ref{thm:derham-local-mhs} by: specializing $w 
\mapsto q$, making a change of variable $q\mapsto q^{-1}$, multiplying by $q^{2mn}$, putting $t\mapsto w$, using (\ref{eq:qbin-pos}), and taking into account the Tate twist. The claim on the Hodge numbers also follows readily.

Now we show the second formula in Theorem \ref{thm:main-weight}. By Theorem \ref{thm:main-generic} (c) and \cite[Corollary 2.26]{mhs}, for all $i$ we have an exact sequence of mixed Hodge structures
\[0 \to H_{i}^{BM}(\ol{O}_p) \to H^{2d_{O_p}-i}(O_p)(d_{O_p}) \to H_{i-1}^{BM}(\ol{O}_{p-1}) \to 0.\]
From the first part, it follows readily that the Hodge numbers of $H^i(O_p)$ are also concentrated on the diagonal, for all $i$. Using the first formula in Theorem \ref{thm:main-weight}, we have (putting $t=qw^{-1}$)
\[\dim_{\C} \op{Gr}^W_j H^{2d_{O_p}-i}(O_p)(d_{O_p}) \cdot q^i w^j = \sum_{s=0}^p w^{p-s} t^{2sm+(p-s)(p-s+2)} {n \choose s}_{t^{2}} {n-1-s \choose p-s}_{t^2} + \]
\[+\sum_{s=0}^{p-1} q \cdot w^{p-1-s} t^{2sm+(p-s-1)(p-s+1)} {n \choose s}_{t^{2}} {n-1-s \choose p-1-s}_{t^2} =\]
\[= \sum_{s=0}^p w^{p-s} \cdot t^{2sm+(p-s)^2} {n \choose s}_{t^{2}} \cdot \left[t^{2(p-s)}{n-1-s \choose p-s}_{t^2} +  {n-1-s \choose p-1-s}_{t^2}  \right].\]
Using the following identities 
\[{a \choose b}_{q}=q^{b}\cdot {a-1 \choose b}_{q}+{a-1 \choose b-1}_{q}, \mbox{ and } \quad {a \choose b}_{q} {a-b \choose c-b}_{q}= {a \choose c}_{q} {c \choose b}_{q},\]
we obtain (after putting $s\mapsto p-s$)
\[\dim_{\C} \op{Gr}^W_j H^{2d_{O_p}-i}(O_p)(d_{O_p}) \cdot q^i w^j = {n \choose p}_{t^{2}}\cdot \sum_{s=0}^p w^{s} \cdot t^{s^2+2(p-s)m} \cdot {p \choose s}_{t^{2}}.\]
Using the Gaussian binomial theorem
\[\prod _{k=0}^{n-1}(1+a^{k}b)=\sum _{k=0}^{n}a^{k(k-1)/2}{n \choose k}_{a}\cdot b^{k},\]
we obtain by putting $a=t^2$ and $b=w \cdot t^{1-2m}$
\[\dim_{\C} \op{Gr}^W_j H^{2d_{O_p}-i}(O_p)(d_{O_p}) \cdot q^i w^j = t^{2pm} \cdot {n \choose p}_{t^{2}}\cdot \prod_{s=0}^{p-1}(1+t^{2s+1-2m}w).\]
We replace $q\mapsto q^{-1}$, let $u=qw$, and multiply both sides with $u^{2d_{O_p}}$ (recall $d_{O_p}=p(m+n-p)$), to get
\[\dim_{\C} \op{Gr}^W_j H^{i}(O_p) \cdot q^i w^j = u^{2p(n-p)} \cdot {n \choose p}_{u^{-2}} \cdot \prod_{s=0}^{p-1}(1+u^{2m-2s-1}w),\]
which, after using (\ref{eq:qbin-pos}), yields the result.
\qed

\begin{remark}\label{rem:totaro}
Let $CH_i(\ol{O}_p)$ denote the Chow groups of $\ol{O}_p$. The determinantal varieties $\ol{O}_p$ are known to be spherical, hence, by a result of Totaro \cite[Theorem 3]{totaro}, the natural cycle map
\[CH_i(\ol{O}_p) \oo \C \longrightarrow W_{-2i}H_{2i}^{BM}(\ol{O}_p)\]
is an isomorphism. Therefore, we recover the (rational) Chow groups computed in \cite{pragacz}*{Section~4} from the first formula in Theorem \ref{thm:main-weight}. More precisely, the summand with $s=p$ in the latter yields exactly the lowest piece of the filtration $W_{-2i}H_{2i}^{BM}(\ol{O}_p)$. Based on these circle of ideas, a conceptual reason as to why the dimension of this agrees with that of the intersection cohomology of $\ol{O}_p$ (computed in \cite{zelevinsky}*{Section~3.3}) would go as follows: in the spectral sequence from Proposition \ref{prop:mhs}, the only $\D$-module composition factor contributing to the lowest pieces $W_{-2i}H_{2i}^{BM}(\ol{O}_p)$ is $D_p$ (appearing in $H_{\ol{O}_p}^{j}(\mc{O}_X)$ only when $j=c_p$, in which case it does so once), whose de Rham cohomology in turn yields the intersection cohomology groups of $\ol{O}_p$.
\end{remark}

As mentioned in the Introduction, the degeneration of the \v{C}ech--de Rham spectral sequence is an open problem in general. We end this section by illustrating that even with the prior knowledge of all the terms on its second page, one can not conclude that the spectral sequence degenerates for weight reasons alone.

\begin{example}\label{ex:theweightisnotenough}
Take $m=n=4$ and $p=2$ in Theorem \ref{thm:derham-local-mhs}, and consider for this case the third page of the \v{C}ech--de Rham spectral sequence from Proposition \ref{prop:mhs}. Then we obtain a differential 
\[\C = \op{Gr}_{16}^W H^{12}_{dR}H^4_{\ol{O}_2}(\mc{O}_X^H) \, \longrightarrow \, \op{Gr}_{16}^W H^{9}_{dR}H^6_{\ol{O}_2}(\mc{O}_X^H) = \C.\]
Hence, this map is between non-trivial spaces of the same weight. We know, a posteriori, that this is zero due to Theorem \ref{thm:main-generic} (c).
\end{example}

\section{The case of skew-symmetric matrices}\label{sec:skew}

In this section we let $X=\bw^2\bb{C}^n$ denote the space of $n\times n$ skew-symmetric matrices, endowed with the natural action of $G=\GL_n$. We let $m=\lfloor{n/2}\rfloor$ and denote the $G$-orbits by $O_p$ as before, where now $O_p$ consists of skew-symmetric matrices of rank $2p$, $0\leq p\leq m$. The goal of this section is to prove the following result, which combined with \eqref{eq:qbin-pos} implies part b) of Theorems~\ref{thm:main2},~\ref{thm:di=0}, and~\ref{thm:main} (as before, we disregard the case $p=m$ when $\ol{O}_p=X$).

\begin{theorem}\label{thm:main-skew}
Suppose that $0\leq p<m=\lfloor n/2\rfloor$, and let $\epsilon=n-2m$.
\begin{itemize} 
\item[(a)] The generating function for de Rham cohomology of local cohomology modules is
\[ \sum_{i,j\geq 0} h^i_{dR}(H^j_{\ol{O}_p}(S))\cdot q^i \cdot w^j  =\sum_{s=0}^p q^{{n\choose 2}-s(2n-2s-1)} \cdot {m \choose s}_{q^4} \cdot w^{2(m-p)^2+p-m+2\epsilon(m-s)} \cdot {m-1-s \choose p-s}_{w^4}.\]
\item[(b)] The Hilbert--Poincar\'e polynomial for the Borel--Moore homology of the orbit closures is given by
\[\sum_{i\geq 0} h_i^{BM}(\ol{O}_p) \cdot q^i \, = \, \sum_{s=0}^p q^{2s(n+\epsilon-1)+(p-s)(2p-2s+3)} \cdot {m \choose s}_{q^{4}} \cdot {m-1-s \choose p-s}_{q^{4}}.\]
\item[(c)] The \v Cech--de Rham spectral sequence \eqref{eq:CdR-sp-seq} degenerates at the $E_2$ page, and the maps $d_i$ in \eqref{eq:les-BM-hom} vanish.
\end{itemize}
\end{theorem}

\subsection{The proof of Theorem~\ref{thm:main-skew}}\label{subsec:proof-skew}

The simple objects in $\opmod_{G}(\D_X)$ are the intersection homology $\D_X$-modules $D_p=\mc{L}(\ol{O}_p, X)$. By \cite{euler}*{Theorem 6.1}, the generating function for the de Rham cohomology of the simples $D_s$ is given by

\begin{equation}\label{eq:deRham-Dp-skew}
\sum_{i\geq 0} h^i_{dR}(D_s) \cdot q^i =  {m \choose s}_{q^4} \cdot q^{{n\choose 2}-s(2n-2s-1)}.
\end{equation}
Moreover, by \cite{raicu-weyman-loccoh}*{(1.4)}, we have for $p<m$ the formal identity
\begin{equation}\label{eq:compos-loccoh-skew}
\sum_{j\geq 0} [H^j_{\ol{O}_p}(S)]\cdot w^j = \sum_{s=0}^p [D_s]\cdot w^{2(m-p)^2+p-m+2\epsilon(m-s)} \cdot {m-1-s \choose p-s}_{w^4}.
\end{equation}
describing the simple $\mc{D}_X$-composition factors of the local cohomology modules $H^j_{\ol{O}_p}(S)$. Combining \eqref{eq:deRham-Dp-skew} with \eqref{eq:compos-loccoh-skew} we obtain the inequality
\begin{equation}\label{eq:ineq-gen-fun-DR-loccoh-skew} 
\sum_{i,j\geq 0} h^i_{dR}(H^j_{\ol{O}_p}(S))\cdot q^i \cdot w^j  \leq\sum_{s=0}^p q^{{n\choose 2}-s(2n-2s-1)} \cdot {m \choose s}_{q^4} \cdot w^{2(m-p)^2+p-m+2\epsilon(m-s)} \cdot {m-1-s \choose p-s}_{w^4}.
\end{equation}
Specializing to $q=w=1$, we obtain
\[\rho^{tot}(\ol{O}_p) \leq N_p := \sum_{s=0}^p {m\choose s}\cdot {m-1-s\choose p-s}.
\]
Using the proof of Lemma~\ref{lem:btot-Op-Np-generic} (with $n$ replaced by $m$, and part (a) of Theorem~\ref{thm:coho-orbit} replaced by part (b)), we get \eqref{eq:btotOp=NpNpm1}, and conclude that \eqref{eq:CdR-sp-seq} degenerates and that the maps $d_i$ in \eqref{eq:les-BM-hom} vanish. Moreover, \eqref{eq:ineq-gen-fun-DR-loccoh-skew} is an equality, and by specializing it to $w=q$ and using the degeneration of \eqref{eq:CdR-sp-seq} and $\dim(X)={n\choose 2}$, we get
\[\sum_{k\geq 0} h_k^{BM}(\ol{O}_p) \cdot q^{n(n-1)-k} = \sum_{i,j\geq 0} h^i_{dR}(H^j_{\ol{O}_p}(S))\cdot q^{i+j}.\]
Making the change of variable $q\to q^{-1}$, multiplying by $q^{n(n-1)}$, and using (\ref{eq:qbin-pos}), we get Theorem~\ref{thm:main-skew}(b).

\subsection{Comparison with Lyubeznik numbers}\label{subsec:Lyub-comparison-skew}

The contrast between the \v Cech-de Rham and Lyubeznik numbers is completely analogous to the discussion in Section \ref{subsec:Lyub-comparison}, and we explain this here briefly. When $n$ is odd, the category $\op{mod}_G(\D_X)$ is semisimple \cite[Theorem 5.7]{lor-wal}, which gives a more direct argument for the inequality in (\ref{eq:ineq-gen-fun-DR-loccoh-skew}) being an equality. Since $\mc{F}(D_p)\cong D_{m-p}$ (e.g. see \cite[Remark 1.5]{raicu-dmods}), the \v Cech--de Rham and Lyubeznik numbers completely determine each other using \eqref{eq:pushpull}. 

We will therefore assume from now on that $n=2m$ is even, when $\op{mod}_G(\D_X)$ is no longer semisimple \cite[Theorem 5.7]{lor-wal}. When $p<m$ the $\D$-module $H^j_{\ol{O}_p}(S)$ can be written as a direct sum of copies of the indecomposable $\D$-modules $Q_0, Q_1, \dots, Q_p$ by \cite{perlman}*{Theorem~1.1}, which are given by
\[Q_m= S_{\Pf}, \qquad Q_p = \dfrac{S_{\Pf}}{\langle \Pf^{2(p-m+1)} \rangle} \quad (0\leq p \leq m-1),\]
where $S_{\Pf}$ denotes the localization of $S$ at the Pfaffian. Note that $Q_0=D_0$ and for $1 \leq p \leq m$, we have the the non-split short exact sequences of $\D$-modules
\begin{equation}\label{eq:ses-Qp-skew}
 0 \lra D_p \lra Q_p \lra Q_{p-1} \lra 0.
\end{equation}
We have a decomposition of $Q_p$ as a $G$-representation
\[ Q_p = \bigoplus_{s=0}^p Q_p^s,\]
where $Q_p^s \simeq D_s$. As a rational $G$-module, the decomposition of $D_p$ is given in \cite[Section 6]{raicu-dmods}. Our conventions are as follows: let $S = \Sym(\bigwedge^2 V)$ with $\dim V=n$, and identify $X=\Spec(S) = \bigwedge^2 V^\vee$ endowed with the action of $G=\GL(V)$. Then
\begin{equation}\label{eq:decomp-Dp-skew}
 D_p = \bigoplus_{\ll \in B(p)}\bb{S}_{\ll} V,
\end{equation}
where 
\begin{equation}\label{eq:def-Bs}
 B(p) \, = \, \{\ll\in\bb{Z}^{n}_{\dom} : \ll_{2p}\geq 2p-n, \, \ll_{2p+1}\leq 2p-n+1, \mbox{ and }\ll_{2i-1}=\ll_{2i} \mbox{ for all } i\}.
\end{equation}

The next two results follow analogously to Lemma \ref{lem:Qps} and Corollary \ref{cor:drQp}.

\begin{lemma}\label{lem:Qps-skew}
 If $\pd\in \bw^2 V^\vee$ is a derivation and $z\in Q_p^s$ then $\pd(z)\in Q_p^s$. 
\end{lemma}
\begin{proof}
We may assume that $z$ belongs to an isotypic component $\bb{S}_{\ll}V$, where $\ll\in B(s)$. Since $\pd\in \bw^2 V^\vee$, it follows from Pieri's rule \cite[Corollary 2.3.5]{weyman} and the fact that $Q_p$ is closed under the action of $\pd$, that
 \[ \pd(z) = \sum_{\nu} z_{\nu},\]
 where for each $\nu$ we have that $z_{\nu}\in \bb{S}_{\nu}V$ for $\nu\in B(t)$ with $t\leq p$, and $\nu$ is obtained from $\ll$ by removing two boxes from the same column, i.e. there exists $r$ with $1\leq r\leq m$ such that
 \[ \nu_{2i-1}=\nu_{2i} = \ll_{2i-1}=\ll_{2i}\mbox{ for }i\neq r,\mbox{ and } \nu_{2r-1}=\nu_{2r}=\ll_{2r-1}-1=\ll_{2r}-1.\]
 Since $z\in\ker(Q_p \onto Q_{s-1})$ it follows that $s\leq t\leq p$. We consider two cases:
 \begin{itemize}
  \item $r\neq s$. Then $\nu_{2s}=\ll_{2s}\geq 2s-n$, and $\nu_{2s+1}\leq\ll_{2s+1}\leq 2s-n+1$. So $\nu\in B(s)$ and $z_{\nu}\in Q_p^s$.
  \item $r=s$. If $\ll_{2s}>2s-n$ then $\nu_{2s}\geq 2s-n$ and $z_{\nu}\in Q_p^s$. If $\ll_{2s}=2s-n$ then $\nu_{2s}=2s-n-1$, so $\nu\not\in B(s)$. Thus, we must have $\nu\in B(t)$ for some $t>s$. However, this implies that
  \[ 2t-n\leq \nu_{2t}\leq \nu_{2s} = 2s-n-1,\]
  which yields $t < s$, a contradiction. \qedhere
 \end{itemize}
\end{proof}

Thus, the non-split exact sequence of $\D$-modules (\ref{eq:ses-Qp-skew}) splits as $\C[\partial_{ij}]$-modules.

\begin{corollary}\label{cor:drQp-skew}
We have a decomposition of complexes 
\[DR(Q_p) = \bigoplus_{s=0}^p DR(D_s).\]
In particular, 
 \[H^i_{dR}(Q_p) = \bigoplus_{s=0}^p H^i_{dR}(D_s).\]
\end{corollary}

We note that the analogues of Remarks \ref{rem:topo} and \ref{rem:dual} hold in the Pfaffian setting as well. 

\section{The case of symmetric matrices}
\label{sec:DeRham-symmetric}

In this section we let $X=\Sym^2\bb{C}^n$ denote the space of $n\times n$ symmetric matrices, endowed with the natural action of $G=\GL_n$. The orbits of the $G$-action on $X$ are denoted by $O_p$, where $O_p$ consists of symmetric matrices of rank $p$, $0\leq p\leq n$. The goal of this section is to prove the following result, which combined with \eqref{eq:qbin-pos} implies part c) of Theorems~\ref{thm:main2},~\ref{thm:di=0}, and~\ref{thm:main} (as before, we disregard the case $p=n$ when $\ol{O}_p=X$).

\begin{theorem}\label{thm:main-symm}
Suppose that $0\leq p<n$, let $m=\lfloor n/2\rfloor$, and define
\[ \epsilon_p = \begin{cases}
1 & \text{if }p\text{ is even and }n=2m+1\text{ is odd}; \\
0 & \text{otherwise}.
\end{cases}
\]
\begin{itemize} 
\item[(a)] The generating function for de Rham cohomology of local cohomology modules is
\[ \sum_{i,j\geq 0} h^i_{dR}(H^j_{\ol{O}_p}(S))\cdot q^i \cdot w^j  =
 \sum_{\substack{s=0 \\ s 
\equiv p \!\!\!\! \pmod{2}}}^p \! q^{{n-s+1\choose 2}} \cdot {m+\epsilon_p \choose 
\lfloor \frac{s}{2} \rfloor}_{q^4}
\cdot w^{1+\binom{n-s+1}{2}-\binom{p-s+2}{2}} \cdot \binom{\lfloor\frac{n-s-1}{2}\rfloor}{\frac{p-s}{2}}_{w^{-4}}.
\]
\item[(b)] The Hilbert--Poincar\'e polynomial for the Borel--Moore homology of the orbit closures is given by
\[
\sum_{i\geq 0} h_i^{BM}(\ol{O}_p) \cdot q^i \, = \, \sum_{\substack{s=0 \\ s 
\equiv p \!\!\!\! \pmod{2}}}^p \!
q^{2\binom{n+1}{2}+\binom{p-s+2}{2}-2\binom{n-s+1}{2}-1} \cdot {m+\epsilon_p \choose 
\lfloor \frac{s}{2} \rfloor}_{q^{-4}} \cdot \binom{\lfloor\frac{n-s-1}{2}\rfloor}{\frac{p-s}{2}}_{q^{4}}.
\]
\item[(c)] The \v Cech--de Rham spectral sequence \eqref{eq:CdR-sp-seq} degenerates at the $E_2$ page, and the maps $d_i$ in \eqref{eq:les-BM-hom} vanish if $n-p$ is even or if $p=1$.
\end{itemize}
\end{theorem}

\subsection{The proof of Theorem~\ref{thm:main-symm}}\label{subsec:proof-symm}

For $p$ with $0\leq p \leq n$, we let $D_p=\mc{L}(\ol{O}_p, X)$ denote the intersection homology $\D$-module corresponding to the trivial local system on the orbit $O_p$. Unlike in the case of general and skew-symmetric matrices, $\opmod_{G}(\D_X)$ contains other simple modules (see \cite{raicu-dmods}*{Theorem~2.9}), but they do not contribute to the local cohomology groups $H^j_{\ol{O}_{p}}(S)$. Indeed, by \cite{raicu-weyman-loccoh}*{(1.5)}, the composition series of local cohomology modules is encoded for $p<n$ by 
\begin{equation}\label{eq:locDpsym}
\sum_{j\geq 0} [H^j_{\ol{O}_{p}}(S)] \cdot w^j = \sum_{\substack{s=0 \\ s 
\equiv p \!\!\!\! \pmod{2}}}^p \! [D_s] 
\cdot w^{1+\binom{n-s+1}{2}-\binom{p-s+2}{2}} \cdot \binom{\lfloor\frac{n-s-1}{2}\rfloor}{\frac{p-s}{2}}_{w^{-4}}.
\end{equation}
Moreover, by \cite[Theorem 5.1]{euler} we have
\begin{equation}\label{eq:deRham-Dp-symm}
\sum_{i\geq 0} h^i_{dR}(D_s) \cdot q^i =  {m+\epsilon_s \choose \lfloor \frac{s}{2} \rfloor}_{q^4} \cdot q^{{n-s+1\choose 2}},
\end{equation}
which combined with \eqref{eq:locDpsym} yields (note that $\epsilon_p=\epsilon_s$ when $s\equiv p \pmod{2}$)
\begin{equation}\label{eq:ineq-gen-fun-DR-loccoh-symm} 
 \sum_{i,j\geq 0} h^i_{dR}(H^j_{\ol{O}_p}(S))\cdot q^i \cdot w^j  \leq
 \sum_{\substack{s=0 \\ s 
\equiv p \!\!\!\! \pmod{2}}}^p \! q^{{n-s+1\choose 2}} \cdot {m+\epsilon_p \choose 
\lfloor \frac{s}{2} \rfloor}_{q^4}
\cdot w^{1+\binom{n-s+1}{2}-\binom{p-s+2}{2}} \cdot \binom{\lfloor\frac{n-s-1}{2}\rfloor}{\frac{p-s}{2}}_{w^{-4}}.
\end{equation}
Specializing to $q=w=1$, we obtain
\[\rho^{tot}(\ol{O}_p) \leq N_p := \sum_{\substack{s=0 \\ s 
\equiv p \!\!\!\! \pmod{2}}}^p \! {m+\epsilon_p \choose 
\lfloor \frac{s}{2} \rfloor} \cdot \binom{\lfloor\frac{n-s-1}{2}\rfloor}{\frac{p-s}{2}}.
\]
It will be useful to extend the above formulas to $p=n$, where
\[  N_n := \rho^{tot}(\ol{O}_n) = \rho^{tot}(X) \overset{\eqref{eq:CdR-for-X}}{=} 1.\]

\begin{lemma}\label{lem:btot-Op-Np-symm}
 If $p\leq n$ and $n-p$ is even, or if $p=1$, then \eqref{eq:btotOp=NpNpm1} holds.
\end{lemma}

\begin{proof}
Suppose first that $p=1$, and note that $N_1=N_0=1$. By Theorem~\ref{thm:coho-orbit}(c) we have $b^{tot}(O_1)=2$, hence \eqref{eq:btotOp=NpNpm1} holds. We therefore assume from now on that $n-2p$ is even. 

Suppose first that $p<n$. If $n=2m$ and $p=2r$ are even, then we have (putting $t=\lfloor s/2 \rfloor$)
\begin{equation}\label{eq:NpNp1-even-case}
N_p+N_{p-1} = \sum_{t=0}^r{m\choose t}\cdot{m-t-1\choose r-t} + 
\sum_{t=0}^{r-1}{m\choose t}\cdot{m-t-1\choose r-t-1} = {m\choose r}\cdot 2^r,
\end{equation}
where the last equality follows from the identity in the proof of Lemma~\ref{lem:btot-Op-Np-generic}. The equality \eqref{eq:btotOp=NpNpm1} now follows from Theorem~\ref{thm:coho-orbit}(c).

If $n=2m+1$ and $p=2r+1$ are odd then we have
\[N_p+N_{p-1} = \sum_{t=0}^r{m\choose t}\cdot{m-t-1\choose r-t} + 
\sum_{t=0}^{r}{m+1\choose t}\cdot{m-t\choose r-t}.\]
Using the fact that ${m+1\choose t} = {m\choose t} + {m\choose t-1}$ and the second equality in \eqref{eq:NpNp1-even-case}, we conclude that
\[ N_p+N_{p-1} = {m\choose r}\cdot 2^r + \sum_{t=0}^r{m\choose t}\cdot{m-t\choose r-t} = {m\choose r}\cdot 2^r + \sum_{t=0}^r{m\choose r}\cdot{r\choose t} = {m\choose r}\cdot 2^r + {m\choose r}\cdot 2^r = {m\choose r}\cdot 2^{r+1}.\]
The equality \eqref{eq:btotOp=NpNpm1} follows again from Theorem~\ref{thm:coho-orbit}(c).

Finally, assume that $p=n$, so that $N_p=1$. If $n=2m$ then
\[N_{p-1}=N_{n-1} = \sum_{t=0}^{m-1}{m\choose t}=2^m-1,\]
hence $N_p+N_{p-1}=2^m = b^{tot}(O_n)$. If $n=2m+1$ then
\[N_{p-1}=N_{n-1} = \sum_{t=0}^{m}{m+1\choose t}=2^{m+1}-1,\]
hence $N_p+N_{p-1}=2^{m+1} = b^{tot}(O_n)$, concluding our proof.
\end{proof}

Since the equality \eqref{eq:btotOp=NpNpm1} implies the degeneration of the spectral sequence \eqref{eq:CdR-sp-seq} for both $\ol{O}_p$ and $\ol{O}_{p-1}$, it follows that in order to prove \eqref{eq:CdR-sp-seq} degenerates for all $p$ it suffices to prove that \eqref{eq:btotOp=NpNpm1} holds for every other value of $p$. This is indeed the case by Lemma~\ref{lem:btot-Op-Np-symm}, hence the first part of Theorem~\ref{eq:locDpsym}(c) holds. It also follows from Lemma~\ref{lem:btot-Op-Np-symm} that \eqref{eq:HOp-leq-hBM} is an equality when $p=1$ or $n-p$ is even, hence as explained in the Introduction, the last conclusion of Theorem~\ref{eq:locDpsym}(c) holds. Parts (a) and (b) of Theorem~\ref{eq:locDpsym} now follow from the fact that \eqref{eq:ineq-gen-fun-DR-loccoh-symm} must be an equality, as in the case of general and skew-symmetric matrices.

\subsection{De Rham cohomology for the modules $Q_p$}\label{subsec:Lyub-comparison-sym}
As mentioned in the Introduction, unlike for general and skew-symmetric matrices, the Lyubeznik numbers of $\ol{O}_p$ are unknown in the symmetric case. Furthermore, the explicit $\D$-module decomposition of the local cohomology modules $H_{\ol{O}_p}^i(S)$ is also not known in general. Nevertheless, we mention some partial results to this end, that lead naturally to the consideration of certain $\D$-modules $Q_p$ analogous to the ones considered in Sections \ref{subsec:Lyub-comparison} and \ref{subsec:Lyub-comparison-skew}.

Due to (\ref{eq:locDpsym}) and \cite[Theorem 5.9]{lor-wal}, when $n-p$ is even ($0\leq p <n$), the $\D$-modules $H_{\ol{O}_p}^i(S)$ are semisimple. In particular, this readily proves in this case that equality holds in (\ref{eq:ineq-gen-fun-DR-loccoh-symm}). Additionally, if $n$ is even (so $p$ is also even), then $\mc{F}(D_p) \cong D_{n-p}$ by \cite[Remark 1.5]{raicu-dmods}, hence the \v Cech-de Rham numbers of $\ol{O}_p$ yield the Lyubeznik numbers of $\ol{O}_{n-p}$ by (\ref{eq:pushpull}). On the other hand, if $n$ is odd (so $p$ is also odd) then $\mc{F}(D_p)$ is a simple equivariant $\D$-module corresponding to a non-trivial local system of an orbit \cite[Remark 1.5]{raicu-dmods}, thus we do not obtain Lyubeznik numbers in this way.

From now on we assume that $n-p$ is odd. We write $S$ for the coordinate ring of $X$, and in order to make the formulas below uniform, we set $D_{n+1}:=S$. For $0\leq p \leq n+1$ (with $n-p$ odd) we consider the following indecomposables $Q_p \in \op{mod}_G(\D_X)$ (cf. \cite[Section 5.3]{lor-wal}):
\[Q_{n+1}= S_{\op{sdet}}, \qquad Q_p = \dfrac{S_{\op{sdet}}}{\langle \op{sdet}^{(p-n+1)/2} \rangle} \quad (0\leq p \leq n-1),\]
where $S_{\op{sdet}}$ denotes the localization of $S$ at the symmetric determinant. We have short exact sequences
\begin{equation}\label{eq:ses-Qp-sym}
 0 \lra D_p \lra Q_p \lra Q_{p-2} \lra 0.
\end{equation}
We note that for $p<n$ (with $n-p$ odd) the $\D$-modules $H_{\ol{O}_p}^i(S)$ are not semisimple in general. In fact, by \cite[Lemma 3.11]{lor-wal} (see also \cite[Lemma 2.4]{binary}) and \cite[Theorem 5.9]{lor-wal}, we have $H_{\ol{O}_p}^{\codim_X O_p}(S) \cong Q_p$. Based on empirical evidence, and on the case of general and skew-symmetric matrices, we conjecture that all the $\D$-modules $H_{\ol{O}_p}^i(S)$ are direct sums of the indecomposables $Q_s$, with $s\leq p$ and $s \equiv p \pmod{2}$.

As in Sections \ref{subsec:Lyub-comparison} and \ref{subsec:Lyub-comparison-skew}, we now show that the non-split exact sequence of $\D$-modules (\ref{eq:ses-Qp-sym}) splits in the category of $\C[\partial_{ij}]$-modules, which via de Rham cohomology gives further indication for the validity of the conjecture due to the fact that equality holds in (\ref{eq:ineq-gen-fun-DR-loccoh-symm}).

\smallskip

We write $S = \Sym(\Sym^2 V)$, so that $X=\Spec(S) = \Sym^2 V^\vee$, with $\dim V = n$. We consider the decomposition of the simple $\D$-modules $D_p$ as a direct sum of irreducible $G$-representations, which is given in \cite{raicu-dmods}*{Theorem~4.1}.
For $0\leq p \leq n+1$ (with $n-p$ odd), we have
\begin{equation}\label{eq:symm-Dp-odd}
 D_p = \bigoplus_{\ll \in \mc{C}(p)}\bb{S}_{\ll} V
\end{equation}
where
\[\mc{C}(p) \, = \, \{\ll\in\bb{Z}^{n}_{\dom}:\ll_i\overset{(\opmod\ 2)}{\equiv} 0\rm{ for }i=1,\cdots,n,\ \ll_{p-1}\geq p-n-1 \geq\ll_{p+1}\}.\]

We have a decomposition of $Q_p$ as a $G$-representation
\[ Q_p = \bigoplus_{\substack{s=0 \\ s 
\equiv p \!\!\!\! \pmod{2}}}^p Q_p^s,\]
where $Q_p^s \simeq D_s$. The next two results are the analogues of Lemma \ref{lem:Qps} and Corollary \ref{cor:drQp}.

\begin{lemma}\label{lem:Qps-sym}
 If $\pd\in \Sym^2 V^\vee$ is a derivation and $z\in Q_p^s$ then $\pd(z)\in Q_p^s$. 
\end{lemma}
\begin{proof}
We can assume that $z$ belongs to an isotypic component $\bb{S}_{\ll}V$, with $\ll\in \mc{C}(s)$. As $\pd\in \Sym^2 V^\vee$, it follows from Pieri's rule \cite[Corollary 2.3.5]{weyman} and the fact that $Q_p$ is closed under the action of $\pd$, that
 \[ \pd(z) = \sum_{\nu} z_{\nu},\]
 where for each $\nu$ we have that $z_{\nu}\in \bb{S}_{\nu}V$ for $\nu\in \mc{C}(t)$ with $t\leq p$ and $t\equiv p \pmod{2}$, and $\nu$ is obtained from $\ll$ by removing two boxes from the same row, i.e. there exists $r$ with $1\leq r\leq n$ such that
 \[ \nu_{i} = \ll_{i}\mbox{ for }i\neq r,\mbox{ and } \nu_{r} = \ll_{r}-2.\]
 Since $z\in\ker(Q_p \onto Q_{s-2})$ it follows that $s\leq t\leq p$. We have two cases:
 \begin{itemize}
  \item $r\neq s-1$. Then $\nu_{s-1}=\ll_{s-1}\geq s-n-1$, and $\nu_{s+1}\leq\ll_{s+1}\leq s-n-1$. So $\nu\in \mc{C}(s)$ and $z_{\nu}\in Q_p^s$.
  \item $r=s-1$. If $\ll_{s-1}>s-n-1$ (and so $\geq s-n+1$) then $\nu_{s-1}\geq s-n-1$ and $z_{\nu}\in Q_p^s$. If $\ll_{s-1}=s-n-1$ then $\nu_{s-1}=s-n-3$, so $\nu\not\in \mc{C}(s)$. Thus, we must have $\nu\in \mc{C}(t)$ for some $t>s$. However, then
  \[t-n-1\leq \nu_{t-1}\leq \nu_{s-1} = s-n-3,\]
  which yields $t < s$, a contradiction. \qedhere
 \end{itemize}
\end{proof}

\begin{corollary}\label{cor:drQp-sym}
We have a decomposition of complexes 
\[DR(Q_p) = \bigoplus_{\substack{s=0 \\ s 
\equiv p \!\!\!\! \pmod{2}}}^p DR(D_s).\]
In particular, 
 \[H^i_{dR}(Q_p) = \bigoplus_{\substack{s=0 \\ s 
\equiv p \!\!\!\! \pmod{2}}}^p H^i_{dR}(D_s).\]
\end{corollary}

Note that the analogue of Remark \ref{rem:topo} holds in the symmetric setting as well.


\section{A related spectral sequence}\label{sec:spec2}

There is a spectral sequence similar to (\ref{eq:CdR-sp-seq}) involving singular cohomology and the local cohomology modules $H_{O_p}^i(\mc{O}_X)$, which also degenerates for most of our matrix orbits $O_p$.

First, consider the more general setting when $X$ is an affine space, and $Z\subset X$ a locally closed irreducible smooth subvariety. Consider the $\D$-module pushforward of the structure sheaf $\mc{O}_Z$ via the map $Z\to \{ pt \}$, which yields singular cohomology (see Introduction). If we factor this map as the composition $Z \to X\setminus\{\ol{Z}\setminus Z\} \to X \to \{pt\}$, we obtain the following spectral sequence of $\D$-modules (cf. \cite[Proposition 1.7.1]{htt})
\begin{equation}\label{eq:locclospec}
E_2^{ij}=H^i_{dR}(H^j_{Z}(\mc{O}_X))  \, \Longrightarrow \, H^{i+j-2 c}(Z),
\end{equation}
where $c=\codim_{X} Z$. Naturally, this is can also be viewed as a spectral sequence of mixed Hodge modules, as we did for the \v{C}ech--de Rham spectral sequence in Section \ref{sec:mhs}.

\begin{proposition}\label{prop:spec2}
With the notation above, assume that the \v{C}ech--de Rham spectral sequence
\[E_2^{ij}=H^i_{dR}(H^j_{Y}(\mc{O}_X)) \Longrightarrow H_{2d_X-i-j}^{BM}(Y)\]
degenerates on the second page for $Y=\ol{Z}$ and $Y=\ol{Z}\setminus Z$, and that the following maps $d_i$ are zero for all $i$:
\[\cdots\lra H_i^{BM}(\ol{Z}\setminus Z)\overset{d_i}{\lra} H_i^{BM}(\ol{Z})\lra H^{2 \dim Z-i}(Z)\lra H_{i-1}^{BM}(\ol{Z}\setminus Z)\overset{d_{i-1}}{\lra}H_{i-1}^{BM}(\ol{Z})\lra\cdots\]
Then the spectral sequence (\ref{eq:locclospec}) also degenerates on the second page, and we have for all $i,j\geq 0$
\[h_{dR}^i(H_{Z}^j(\mc{O}_X)) = h_{dR}^i(H_{\ol{Z}}^j(\mc{O}_X))+ h_{dR}^i(H_{\ol{Z}\setminus Z}^{j+1}(\mc{O}_X)).\]
\end{proposition}

\begin{proof}
Consider the long exact sequence in local cohomology corresponding to the inclusion $Z \subset \ol{Z}$:
\begin{equation}\label{eq:longloc}
\cdots \to H^i_{\ol{Z}\setminus Z}(\mc{O}_X) \xrightarrow{d^i} H^i_{\ol{Z}}(\mc{O}_X) \to H^i_Z(\mc{O}_X) \to H^{i+1}_{\ol{Z}\setminus Z}(\mc{O}_X) \xrightarrow{d^{i+1}} H^{i+1}_{\ol{Z}}(\mc{O}_X) \to \cdots
\end{equation}
In particular, we have for all $i,j\geq 0$ \[h_{dR}^i(H_{Z}^j(\mc{O}_X)) \leq h_{dR}^i(H_{\ol{Z}}^j(\mc{O}_X))+ h_{dR}^i(H_{\ol{Z}\setminus Z}^{j+1}(\mc{O}_X)).\]
Summing these up for all $i,j$,  the spectral sequence (\ref{eq:locclospec}) together with the degeneration of the two \v{C}ech-de Rham spectral sequences gives
\[b^{tot}(Z) \leq b^{BM}_{tot}(\ol{Z}) + b^{BM}_{tot}(\ol{Z}\setminus Z).\]
Now the vanishing of the maps $d_i$ implies that equality holds in all of the above inequalities (cf. also (\ref{eq:btotOp-leq-btotBM})), and that the spectral sequence (\ref{eq:locclospec}) degenerates as claimed.
\end{proof}

In the case of our matrix orbits $O_p$, Proposition \ref{prop:spec2} together with Theorems \ref{thm:di=0} and \ref{thm:main} readily yields the following result.

\begin{corollary}
When $Z=O_p$, the spectral sequence (\ref{eq:locclospec}) degenerates on the second page in all of the cases from Theorem \ref{thm:di=0}.
\end{corollary}

While the claim about the de Rham cohomology of $H^i_{O_p}(\mc{O}_X)$ from Proposition \ref{prop:spec2} is also valid in the cases above, we can show a sharper claim about these local cohomology modules as follows.

Due to parity reasons, we see from the formulas (\ref{eq:compos-loccoh-generic}), (\ref{eq:compos-loccoh-skew}), and (\ref{eq:locDpsym}) that the $\D$-modules $H^i_{\ol{O}_{p-1}}(\mc{O}_X)$ and $H^i_{\ol{O}_{p}}(\mc{O}_X)$ have no common composition factors. Thus, the maps $d^i$ in (\ref{eq:longloc}) (with $Z=O_p$) are all zero, and the long exact sequence breaks up into short exact sequences
\[0 \to H_{\ol{O}_p}^i(\mc{O}_X) \to H^i_{O_p}(\mc{O}_X) \to H^{i+1}_{\ol{O}_{p-1}}(\mc{O}_X) \to 0.\]
We claim that these exact sequences of equivariant $\D$-modules split. For the case of general matrices, this follows from \cite[Theorem 6.1 and Lemma 6.5]{lor-rai} and \cite[Theorem 5.4]{lor-wal}, for skew-symmetric matrices from \cite[Theorem 1.1]{perlman} and \cite[Theorem 5.7]{lor-wal}, and for symmetric matrices due to parity reasons by the formula (\ref{eq:locDpsym}) and \cite[Theorem 5.9]{lor-wal}.
Hence, for all $i\geq 0$ and $p\geq 1$ we have (as $\D_X$-modules)
\begin{equation}
    H^i_{O_p}(\mc{O}_X) \,\, \cong \,\, H_{\ol{O}_p}^i(\mc{O}_X) \oplus H^{i+1}_{\ol{O}_{p-1}}(\mc{O}_X),
\end{equation}
which is much stronger than the second claim in Proposition \ref{prop:spec2}.

\section*{Acknowledgements}
 Experiments with the computer algebra software Macaulay2 \cite{M2} have provided numerous valuable insights. Raicu acknowledges the support of the National Science Foundation Grant No.~1901886.

\end{document}